\documentclass[12pt]{article}
 \usepackage{amsmath,amssymb,amscd,amsthm,esint}
\usepackage{graphics,amsmath,amssymb,amsthm,mathrsfs}

\newtheorem{theorem}{Theorem}[section]

\newtheorem{lemma}[theorem]{Lemma}

\newtheorem{remark}[theorem]{Remark}

\def\xxint#1#2#3{{\setbox0=\hbox{$#1{#2#3}{\int}$}
  \vcenter{\hbox{$#2#3$}}\kern-.5\wd0}}

\def \RD {{\mathbb R}^d}

\def\varep{\varepsilon}

\newcommand{\average}{-\!\!\!\!\!\!\int}

\begin{document}

\title
{\bf Uniform Regularity Estimates \\ in Parabolic Homogenization}

\author{Jun Geng\thanks{Supported in part by the NNSF of China (11201206) and
Fundamental Research Funds for the Central Universities (LZUJBKY-2013-7).}
\and Zhongwei Shen\thanks{Supported in part by NSF grant DMS-1161154.}}

\date{ }

\maketitle

\begin{abstract}
We consider a family of second-order parabolic systems in divergence form
 with rapidly oscillating and time-dependent coefficients, arising in the theory of homogenization.
We obtain uniform interior $W^{1,p}$, H\"older, and Lipschitz estimates as well  as boundary $W^{1,p}$ and
H\"older estimates, using compactness methods.
As a consequence, we  establish uniform $W^{1,p}$ estimates for the initial-Dirichlet problems 
in $C^{1}$ cylinders.



 \end{abstract}

 \medskip


\section {Introduction and statement of main results}
\setcounter{equation}{0}

The primary purpose of this paper is to study uniform regularity estimates
for a family of second-order parabolic systems in divergence form with rapidly oscillating and time-dependent coefficients, 
arising in the theory of homogenization.
We obtain uniform interior $W^{1,p}$, H\"older, and Lipschitz estimates as well as boundary $W^{1,p}$ and
H\"older estimates, using  compactness methods.
We also establish uniform $W^{1,p}$ estimates for  the
initial-Dirichlet problems in $C^{1}$ cylinders.

More precisely, we consider the family of the parabolic systems
\begin{equation}\label{1-1}
\big(\partial_t +\mathcal{L}_\varep\big) u_\varep =F, \quad \varep>0,
\end{equation}
where
\begin{equation}\label{1-2}
\mathcal{L}_\varepsilon=-{\rm div}\left[A\big({x}/{\varepsilon}, {t}/{\varep^2}\big)\nabla\right]
=
-\frac{\partial}{\partial x_i}
\Big[a_{ij}^{\alpha\beta}\left(\frac{x}{\varepsilon}, \frac{t}{\varep^2}\right)
\frac{\partial}{\partial x_j}\Big]
\end{equation}
(the summation convention on repeated indices is used throughout the paper).
We will assume that the coefficient matrix  
$A(y,s)=\big(a_{ij}^{\alpha\beta}(y,s)\big)$ with $1\le i, j\le d$ and $1\le \alpha, \beta\le m$
is  real,  bounded, measurable, and satisfies the ellipticity condition,
\begin{equation}\label{e1.3}
\mu|\xi|^2\leq a_{ij}^{\alpha\beta}(y,s)\xi_i^{\alpha}\xi_j^\beta\leq \frac{1}{\mu}|\xi|^2 
            \qquad  \text{\rm for any }  \ (y,s)\in \mathbb{R}^{d+1}, \xi=(\xi_i^\alpha)\in \mathbb{R}^{dm},
\end{equation}
where $\mu>0$,
 and the periodicity condition,
\begin{align}\label{e1.4}
A(y+z, t+s)=A(y,t)  \qquad {\rm for} \ (y,t)\in \mathbb{R}^{d+1} \ {\rm and}\ (z,s)\in {\mathbb Z}^{d+1}.
\end{align}
Some  smoothness conditions on $A$ are also needed.

For $(x_0, t_0)\in \mathbb{R}^{d+1}$ and $r>0$, let
$$
Q_r (x_0, t_0) = B(x_0, r) \times (t_0-r^2, t_0),
$$
where $B(x_0, r)=\big\{ x\in \mathbb{R}^d: \, |x-x_0|<r\big\}$.
If $Q=Q_r(x_0, t_0)$ and $\alpha>0$, we will use $\alpha Q$ to denote $Q_{\alpha r} (x_0, t_0)$.
The following two theorems give the interior $W^{1,p}$,
H\"older, and Lipschitz  estimates, which are dilation invariant and uniform in $\varep>0$.

\begin{theorem}\label{main-theorem-1}
Let $A=A(x,t)$ be a matrix satisfying the ellipticity and periodicity conditions (\ref{e1.3})-(\ref{e1.4}).
Assume that $A\in VMO_x$.
Let $u_\varep$ be a weak solution of $\big(\partial_t +\mathcal{L}_\varep\big) u_\varep =\text{\rm div} (f)$
in $2Q$, where $Q=Q_r(x_0, t_0)$ and $f=(f_i)\in L^p (2Q)$ for some $2<p<\infty$.
Then
\begin{equation}\label{1.1-1}
\left(\average_{Q} |\nabla u_\varep|^p\right)^{1/p}
\le C_p \left\{ \frac{1}{r} \left(\average_{2Q} |u_\varep|^2\right)^{1/2}
+\left(\average_{2Q} |f|^p \right)^{1/p} \right\},
\end{equation}
where $C_p$ depends at most on $d$, $m$, $p$, and $A$.
Moreover, if $p>d+2$ and $\alpha =1-\frac{d+2}{p}$, then
\begin{equation}\label{1.1-2}
\| u_\varep\|_{C^{\alpha, \alpha/2} (Q)}
\le  
C \, r^{1-\alpha} \left\{ \frac{1}{r} \left(\average_{2Q} | u_\varep|^2\right)^{1/2}
+\left(\average_{2Q} |f|^p \right)^{1/p} \right\}.
\end{equation}
\end{theorem}

\begin{theorem}\label{main-theorem-2}
Let $A=A(x,t)$ be a matrix satisfying the ellipticity and periodicity conditions (\ref{e1.3})-(\ref{e1.4}).
Assume that $A$ is H\"older continuous in $(x,t)$.
Let $u_\varep$ be a weak solution of $\big(\partial_t +\mathcal{L}_\varep\big) u_\varep =F$
in $2Q$, where $Q=Q_r(x_0, t_0)$ and $F\in L^p (2Q)$ for some $p>d+2$.
Then
\begin{equation}\label{1.2-1}
\| u_\varep\|_{C^{1, 1/2}(Q)}
\le C_p\left\{ \frac{1}{r} \left( \average_{2Q} |u_\varep|^2\right)^{1/2}
+r \left(\average_{2Q} |F|^p\right)^{1/p} \right\},
\end{equation}
where $C_p$ depends at most on $d$, $m$, $p$, and $A$.
\end{theorem}

We refer the reader to Section 2 for the definition of $VMO_x$, and point out here that
if $A=A(x,t)$ is uniformly continuous in $x$ and measurable in $t$, then $A\in VMO_x$.
In Theorems \ref{main-theorem-1} and \ref{main-theorem-2} we also have used the notation:
\begin{equation}\label{1-5}
\aligned
\| u\|_{C^{\alpha, \alpha/2} (E)}
& =\sup_{\substack{ (x,t), (y,s)\in E \\ (x,t)\neq (y,s)}}
 \frac{|u(x,t)-u (y,s)|}{(|x-y| +|t-s|^{1/2})^\alpha }, \quad 0<\alpha\le 1,\\
 \average_E u & =\frac{1}{|E|}\int_E u.
 \endaligned
 \end{equation}
 
 To describe the boundary estimates, we need to introduce more notation.
 Let $\Omega$ be a bounded domain in $\mathbb{R}^d$.
 For $x_0\in \overline{\Omega}$, $t_0\in \mathbb{R}$, and $0<r<r_0=\text{diam}(\Omega)$, we let
 \begin{equation}\label{1-6}
 \aligned
& \Omega_r (x_0,t_0) = \big[ B(x_0, r) \cap \Omega \big] \times (t_0-r^2, t_0),\\
& \Delta_r (x_0, t_0) =\big[  B(x_0, r)\cap \partial\Omega \big] \times (t_0-r^2, t_0).
\endaligned
\end{equation}
The next theorem provides the uniform boundary $W^{1, p}$ and H\"older estimates.

\begin{theorem} \label{main-theorem-3}
Assume that $A$ satisfies the same conditions as in Theorem \ref{main-theorem-1}.
Let $2<p<\infty$ and $\Omega$ be a bounded $C^{1}$ domain in $\mathbb{R}^d$. Suppose that
\begin{equation}\label{1.3-0}
\big(\partial_t +\mathcal{L}_\varep\big) u_\varep =\text{\rm div} (f) \quad \text{ in } \Omega_{2r} (x_0, t_0) \quad \text{ and } \quad
u_\varep =0 \quad \text{ on } \Delta_{2r} (x_0, t_0),
\end{equation}
for some $x_0\in \partial\Omega$ and $t_0\in \mathbb{R}$, where $f=(f_i)\in L^p(\Omega_{2r} (x_0, t_0))$.
 Then
\begin{equation}\label{1.3-1}
\left(\average_{\Omega_r (x_0, t_0)} |\nabla u_\varep|^p\right)^{1/p}
\le C_p \left\{ \frac{1}{r} \left(\average_{\Omega_{2r} (x_0, t_0)} |u_\varep|^2\right)^{1/2}
+\left(\average_{\Omega_{2r} (x_0, t_0)} |f|^p \right)^{1/p} \right\},
\end{equation}
where $C_p$ depends at most on $d$, $m$, $p$, $A$, and $\Omega$.
Moreover, if $p>d+2$ and $\alpha =1-\frac{d+2}{p}$, then
\begin{equation}\label{1.3-2}
\| u_\varep\|_{C^{\alpha, \alpha/2} (\Omega_r (x_0, t_0))}
\le C r^{1-\alpha}
\left\{ \frac{1}{r} \left(\average_{\Omega_{2r} (x_0, t_0)} |u_\varep|^2\right)^{1/2}
+\left(\average_{\Omega_{2r} (x_0, t_0)} |f|^p \right)^{1/p} \right\}.
\end{equation}
\end{theorem}

As we indicated earlier, the family of operators $\{ \mathcal{L}_\varep\}$ arises
in the theory of homogenization. Indeed, consider the initial-Dirichlet problem
\begin{equation}\label{Dirichlet-problem}
\left\{
\aligned
\big( \partial_t+\mathcal{L}_\varep \big) u_\varep & =F & \quad & \text{ in } \Omega\times (0,T),\\
 u_\varep & =0 & \quad & \text{ on } \partial\Omega\times (0,T),\\
 u_\varep  & =0 & \quad &\text{ on } \Omega \times \{ 0\}.
\endaligned
\right.
\end{equation}
Under the ellipticity condition (\ref{e1.3}), it is well known that for any  $F
 \in L^2(0,T; W^{-1,2}(\Omega))$,
 (\ref{Dirichlet-problem}) has a unique weak solution 
in $L^2(0, T; W^{1,2}_0(\Omega))$. Moreover,
the solution $u_\varep$ of (\ref{Dirichlet-problem}) satisfies
\begin{equation}\label{e1.5}
\| u_\varep\|_{L^2(0,T; W^{1,2}_0(\Omega))}
\le C \,
\| F\|_{L^2(0,T; W^{-1,2}(\Omega))},
\end{equation}
where $C$ depends only on $d$, $m$, $\mu$, $\Omega$, and $T$.
With the additional periodicity condition (\ref{e1.4}), it follows from the theory of homogenization that
as $\varep\to 0$, 
$u_\varep \to u_0$ weakly in $L^2(0, T; W_0^{1,2}(\Omega))$ and strongly in $L^2(\Omega\times (0,T))$.
Furthermore, the limiting function $u_0$ is a solution of the initial-Dirichlet problem in $\Omega\times (0,T)$
for some parabolic system with
constant coefficients (see e.g. \cite[pp.140-142]{bensoussan-1978}).

Uniform regularity estimates play an important role in the study of convergence problems in homogenization.
In the elliptic case, where $\mathcal{L}_\varep =-\text{div} \big(A(x/\varep)\nabla\big)$,
the interior Lipschitz estimates, as well as the boundary Lipschitz estimates with Dirichlet conditions
in $C^{1,\alpha}$ domains, were established by
M. Avellaneda and F. Lin  in \cite{AL-1987},
under the ellipticity, periodicity, and H\"older continuity conditions on $A=A(y)$.
Also see related work in \cite{AL-1987-ho, AL-1989-ho,
AL-1989-II, AL-1991, Caffarelli-Peral, Shen-2008, Kenig-Shen-1, Kenig-Shen-2, Kenig-Lin-Shen-1,
Geng-Shen-Song, Shen-Song} for various uniform  estimates in elliptic homogenization.
In particular, under the additional symmetry condition $A^*=A$,
 the boundary Lipschitz estimates with Neumann boundary conditions
in $C^{1, \alpha}$ domains were recently obtained 
by C. Kenig, F. Lin, and Z. Shen in \cite{Kenig-Lin-Shen-1}.
We point out that the Lipschitz estimates in \cite{AL-1987, Kenig-Lin-Shen-1} and in our Theorem \ref{main-theorem-2}
 are sharp in the sense that $\nabla u_\varep$ in general are not uniformly H\"older continuous.
We also mention that uniform regularity estimates in \cite{AL-1987, Kenig-Shen-1, Kenig-Shen-2, Kenig-Lin-Shen-1}
have been used to establish sharp rates of convergence of solutions and eigenvalues 
(see e.g.  \cite{AL-1987, Kenig-Lin-Shen-2, Kenig-Lin-Shen-3, Kenig-Lin-Shen-4}). 

Homogenization of parabolic equations and systems has many applications in mechanics and physics.
This paper represents our attempt to extend the results in \cite{AL-1987} to the
parabolic setting, where the elliptic operator $-\text{div} \big(A(x/\varep)\nabla \big)$
is replaced by the parabolic operator $\partial_t -\text{div} \big(A(x/\varep, t/\varep^2)\nabla\big)$.
To this end our first observation is that by a real-variable argument originated in \cite{Caffarelli-Peral} and further developed in \cite{Shen-2005-bounds}, one may reduce the $W^{1,p}$ estimates (\ref{1.1-1}) and
(\ref{1.3-1}) for
$\big(\partial_t +\mathcal{L}_\varep\big) u_\varep =\text{div} (f)$ to a weak reverse
H\"older inequality for local solutions of $\big(\partial_t +\mathcal{L}_\varep \big) u_\varep =0$.
As for the Lipschitz estimate (\ref{1.2-1}), the nonhomogenous case may be reduced to the homogenous case by
using the matrix of fundamental solutions.

 Our main tool for studying local solutions of $\big(\partial_t +\mathcal{L}_\varep\big) u_\varep =0$
 is a three-step compactness argument, similar to that used by Avellandeda and Lin in \cite{AL-1987}.
 The first step uses the fact that if $\varep_k \to 0$ and $\{ A_k\}$ is a sequence of 
 matrices satisfying (\ref{e1.3})-(\ref{e1.4}), then $\partial_t +\mathcal{L}^k_{\varep_k}=
 \partial_t -\text{div} \big( A_k (x/\varep_k, t/\varep_k^2)\nabla\big)$
 G-converges to a second-order parabolic operator with constant coefficients (see 
 Theorem \ref{homo-theorem-2}), whose solutions possess much better regularity properties.
 The second step is an iteration process and relies on the following rescaling property
 of $ \partial_t +\mathcal{L}_\varep$:
 \begin{equation}\label{rescaling}
 \aligned
& \text{ if } \big(\partial_t +\mathcal{L}_\varep\big) u_\varep =F \text{ and } v(x,t)=u_\varep (\delta x, \delta^2 t),\\
& \text{ then } \left(\partial_t +\mathcal{L}_{\varep/\delta} \right) v= G, \text{ where } G(x,t)=\delta^2 F (\delta x, \delta^2 t).
\endaligned
 \end{equation}
 The last step is a blow-up argument and uses the local  regularity theory 
 for the operator $\partial_t +\mathcal{L}_1$.
 We remark that the desired local regularity estimates are classical if $A(x,t)$ is H\"older continuous
 in $(x,t)$.
 There has been a lot of work on $W^{1,p}$ estimates for elliptic and parabolic operators with
 discontinuous coefficients (see e.g.\,\cite{ Auscher, Byun-Wang-2004, Byun-2005,
Shen-2005-bounds, Byun-2007, Krylov-2007, Kim-Krylov-2007, Byun-Wang-2008, Dong-Kim-2010, Dong-2010} and their references). 
In particular, under the assumption $A\in VMO_x$, the interior and boundary
 $W^{1,p}$ estimates for $\partial_t +\mathcal{L}_1$, which are used in the proof of Theorems 
 \ref{main-theorem-1} and \ref{main-theorem-3}, were established by S. Byun \cite{Byun-2007}
 and N.V. Krylov \cite{Krylov-2007}.
 
Finally,  as a corollary of Theorems \ref{main-theorem-1} and \ref{main-theorem-3},
we also obtain the uniform estimates in $W^{1,p}$ of solutions to the initial-Dirichlet problem (\ref{Dirichlet-problem})
for $1<p<\infty$.

\begin{theorem}\label{main-theorem-4}
Assume that $A$ satisfies the same conditions as in Theorem \ref{main-theorem-1}.
Let $1<p<\infty$ and $\Omega$ be a bounded $C^{1}$ domain in $\RD$.
Then for any  $F\in L^p(0,T; W^{-1, p}(\Omega))$, the unique solution to the
 initial-Dirichlet  problem (\ref{Dirichlet-problem}) in $L^p(0,T; W^{1,p}_0(\Omega))$
  satisfies 
  \begin{equation}\label{L-p-estimate}
\| u_\varep\|_{L^p(0,T; W^{1,p}_0 (\Omega))}
\le C_p\,
\| F\|_{L^p (0, T; W^{-1,p}(\Omega))},
\end{equation}
where $C_p$ depends at most on $d$, $m$, $p$, $A$, $T$, and $\Omega$.
\end{theorem}


\section{Weak solutions of parabolic systems}
\setcounter{equation}{0}

In this section we recall some properties of weak solutions of the parabolic system
\begin{equation}\label{2.1}
\partial_t u -\text{div }\big( A(x,t)\nabla u\big) =F\qquad \text{ in } \Omega \times (T_0, T_1),
\end{equation}
where the matrix $A$ satisfies the ellipticity condition (\ref{e1.3}) and $\Omega$ is a bounded Lipschitz domain in $\mathbb{R}^d$.

Let $F\in L^2 (T_0, T_1; W^{-1,2} (\Omega))$.
We call $u$ a weak solution of (\ref{2.1}) if $u\in L^2(T_0, T_1; W^{1,2}(\Omega))$ and 
\begin{equation}
-\int_{T_0}^{T_1} \int_\Omega u \cdot \frac{\partial\phi}{\partial t}\, dx dt
+\int_{T_0}^{T_1}\int_\Omega A\nabla u \cdot \nabla \phi\, dx dt
=\int_{T_0}^{T_1} <F, \phi>\, dt,
\end{equation}
for any $\mathbb{R}^m$-valued function $\phi$ in $C_0^\infty (\Omega \times (T_0, T_1))$,
where $<\, , \, >$ denotes the pairing between $W^{1,2}_0(\Omega)$ and its dual $W^{-1, 2}(\Omega)$.

The following two lemmas will be useful to us.

 \begin{lemma} \label{Cacci-Lemma}
Let $u$ be a weak solution of $\partial_t  u -\text{\rm div} (A(x,t)\nabla u)=F+\text{\rm div} (f)$ in
$2Q$, where $Q=Q_{r} (x_0, t_0)$ and $f=(f_i), F\in L^2(2Q)$.
Then $u\in C ((t_0-r^2, t_0); L^2(B(x_0, r)))$ and
\begin{equation}\label{Caccio}
\aligned
\sup_{t\in (t_0-r^2, t_0)} \int_{B(x_0, r)} &  | u(x,t)|^2\, dx
  +\int_Q |\nabla u|^2\, dxdt \\
 & \le C \left\{ \frac{1}{r^2}
\int_{2Q} |u|^2 +\int_{2Q} |f|^2  + r^2 \int_{2Q} |F|^2 \right\},
\endaligned
\end{equation}
where $C$ depends only on $d$, $m$, and $\mu$.
\end{lemma}

\begin{proof}
This is a local energy estimate for the parabolic system with bounded measurable coefficients.
 See \cite[Section III.2]{Lady} for a proof.
\end{proof}

\begin{lemma}\label{Poincare-lemma}
Let $u$ be a weak solution of $\partial_t u -\text{\rm div} \big(A(x,t)\nabla u\big)=\text{\rm div} (f)$ in $2Q$,
where $Q =Q_r (x_0, t_0)$ and  $f=(f_i)\in L^2(2Q)$.
Then
\begin{equation}\label{Poincare}
\int_{Q}
|u- \average_{Q} u |^2
\le C\, r^2\left\{  \int_{2Q} |\nabla u|^2 +\int_{2Q} |f|^2 \right\},
\end{equation}
where $C$ depends only on $d$, $m$, and $\mu$.
\end{lemma}

\begin{proof} The proof is similar to  that of \cite[Lemma 3]{Struwe-1981}.
\end{proof}

The next lemma follows from the Campanato's characterization of (parabolic) H\"older spaces, whose proof
may be found in \cite{Lieberman}.

\begin{lemma}\label{Cam-Lemma}
Let $u\in L^2 (2Q)$ for some $Q=Q_r(x_0,t_0)$.
Suppose that there exist $\alpha\in (0,1]$ and $N>0$ such that
$$
\average_{Q_\rho (x,t)} |u - \average_{Q_\rho (x,t)} u |^2
\le N^2 \rho^{2 \alpha} \qquad 
\text{ for any } (x,t)\in Q \text{ and } 0<\rho<r/2.
$$
Then 
$$
\| u\|_{C^{\alpha, \alpha/2} (Q)}
\le C\, N,
$$
where $C$ depends only on $d$ and $\alpha$.
\end{lemma}

In the case of scalar equations ($m=1)$, it follows from the classical De Giorgi-Nash-Moser estimates that weak solutions
 of (\ref{2.1}) are H\"older continuous of order $\alpha$ for some $\alpha>0$. If $m>1$, this is no longer true for all dimensions, and some additional smoothness condition on $A$ is needed in higher dimensions.
 In particular, it is well known that if $A$ is H\"older continuous in $(x,t)$, then
 $\nabla u$ is locally H\"older continuous.
 However, interior $W^{1,p}$ and H\"older estimates for $u$ hold under much weaker conditions on $A$.

Let $A=A(x,t)$ be a locally integrable function in $\mathbb{R}^{d+1}$.
We say $A\in VMO_x$ if
$$
\lim_{r\to 0} A^\# (r) =0,
$$
where
\begin{equation}\label{A-sharp}
A^\# (r)
=\sup_{\substack{ 0<\rho<r\\ (x,t)\in \mathbb{R}^{d+1}}}
\frac{1}{\rho^2 |B(x,\rho)|^2}
\int_{t}^{t+\rho^2}\int_{y\in B(x,\rho)} \int_{z\in B(x,\rho)}
|A(y,s)-A(z,s)|\, dydz ds.
\end{equation}
It follows from \cite{Byun-2007, Krylov-2007} that if $A\in VMO_x$ and $u$ is a weak solution of
$\partial_t u -\text{div} \big( A(x,t)\nabla u\big) =\text{div}(f)$ in $2Q$ for some $Q=Q_r (x_0, t_0)$ and
$0<r<1$, then
\begin{equation}\label{krylov}
\left(\average_{(7/4)Q} |\nabla u|^p \right)^{1/p}
\le {C_p}\left\{ \frac{1}{r}
\left(\average_{2Q} |u|^2\right)^{1/2}
+\left( \average_{2Q} |f|^p\right)^{1/p}
 \right\},
\end{equation}
where $2<p<\infty$ and $C_p$ depends only $d$, $m$, $p$, and $A$.
Note that if $Q_\rho (x,t)\subset (3/2)Q$, then
$$
\left(\average_{Q_\rho (x,t)} |u -\average_{Q_\rho (x,t)} u|^2\right)^{1/2}
\le C r \left( \frac{\rho}{r} \right)^{1-\frac{d+2}{p}}
\left(\average_{(7/4)Q}\big[ |\nabla u|^p +|f|^p \big]\right)^{1/p},
$$
where we have used Lemma \ref{Poincare-lemma} and H\"older's inequality.
This, together with (\ref{krylov}) and Lemma \ref{Cam-Lemma}, shows that if $p>d+2$ and
$\alpha =1-\frac{d+2}{p}$, then
\begin{equation}\label{krylov-1}
\| u\|_{C^{\alpha, \alpha/2} (Q)}
\le C r^{1-\alpha}
\left\{ \frac{1}{r} \left(\average_{2Q} |u|^2\right)^{1/2}
+\left( \average_{2Q} |f|^p \right) \right\}^{1/p}.
\end{equation}

\begin{remark}\label{remark-2.1}
{\rm 
Suppose that
\begin{equation}\label{VMO-x}
A^\# (r)\le \omega (r) \quad \text{ for } 0<r<1,
\end{equation}
where $\omega(r)$ is a bounded increasing function on $(0,1)$ such that $\lim_{r\to 0^+} \omega (r)=0$.
It was in fact proved in \cite{Byun-2007, Krylov-2007} that the constant $C$ in (\ref{krylov}) depends only on 
$d$, $m$, $p$, $\mu$, and the function  $\omega(r)$.
This will be important to us.
Indeed, let $A_\varep (x,t)=A(x/\varep, t/\varep^2)$.
It is easy to see that if $\varep\ge \varep_0>0$,
$$
\big( A_\varep \big)^\# (r) =A^\# (r/\varep)\le A^\# (r/\varep_0).
$$
As a result, if $\varep\ge \varep_0>0$ and
$\big( \partial_t +\mathcal{L}_\varep \big) u_\varep =\text{div} (f)$ in $2Q$ for some $Q=Q_r (x_0,t_0)$
and $0<r<1$, then $u_\varep$ satisfies the estimates (\ref{krylov})-(\ref{krylov-1})
with constant $C$ depending only on $d$, $m$, $p$, $\mu$, $\varep_0$, and $\omega (r)$.
Therefore we will only need to treat the case where $\varep$ is sufficiently small.
}
\end{remark}


\section{Homogenization of parabolic systems}
\setcounter{equation}{0}

Throughout this section we will assume that $\mathcal{L}_\varep =-\text{div} \big( A(x/\varep, t/\varep^2)\nabla\big)$ with 
coefficient matrix $A(y,s)=\big(a_{ij}^{\alpha\beta} (y,s)\big)$ satisfies the ellipticity condition (\ref{e1.3}) and periodicity condition (\ref{e1.4}).
No additional smoothness condition is needed.

Let $Y=[0,1)^{d+1}$. A function $h=h(y,s)$ is said to be $Y$-periodic if $h$ is periodic with respect to $\mathbb{Z}^{d+1}$.
For each $1\le j\le d$ and $1\le \beta \le m$, let 
$\chi_j^\beta =(\chi_j^{1\beta}, \dots, \chi_j^{m\beta})$ be the weak solution of the following cell problem:
\begin{equation}\label{cell-problem}
\left\{
\aligned
& \frac{\partial \chi_j^\beta}{\partial s}
+\mathcal{L}_1 (\chi_j^\beta)  =-\mathcal{L}_1 (P_j^\beta) \quad \text{ in } \mathbb{R}^{d+1},\\
& \chi^\beta_j =\chi_j^\beta (y,s) \text{ is $Y$-periodic},\\
& \int_Y \chi_j^\beta =0,
\endaligned
\right.
\end{equation}
where $P_j^\beta =y_j (0, \dots, 1, \dots, 0)$ with $1$ in the $\beta^{th}$ position.
The matrix $\chi=\chi (y,s) =\big(\chi_j^{\alpha\beta} (y,s)\big)$ is called the matrix of correctors for 
$\partial_t + \mathcal{L}_\varep$ in $\mathbb{R}^{d+1}$.
Note that by (\ref{cell-problem}),
$$
\big(\partial_s+\mathcal{L}_1 \big) \big( \chi_j^\beta + P_j^\beta\big) = 0 \quad \text{ in } \mathbb{R}^{d+1}.
$$
It follows from the rescaling property (\ref{rescaling}) that
\begin{equation}\label{corrector-equation}
\big( \partial_t +\mathcal{L}_\varep \big)
\left\{ \varep \chi_j^\beta (x/\varep, t/\varep^2) +P_j^\beta (x)\right\}=0 \quad \text{ in } \mathbb{R}^{d+1}.
\end{equation}
 
Let $\widehat{A}=(\widehat{a}_{ij}^{\alpha\beta})$, where $1\le i,j\le d$, $1\le \alpha, \beta\le m$, and
\begin{equation}\label{A-hat}
\aligned
\widehat{a}_{ij}^{\alpha\beta}
&=\int_Y
\left[ a_{ij}^{\alpha\beta} +a_{i\ell}^{\alpha\gamma} \frac{\partial}{\partial y_\ell}
\left( \chi_j^{\gamma\beta} \right) \right]\\
&
=\int_Y A\nabla \left\{ P_j^\beta +\chi_j^\beta\right\} \cdot \nabla P_i^\alpha.
\endaligned
\end{equation}
Define $\mathcal{L}_0=-\text{div} (\widehat{A}\nabla )$.
Then $\partial_t +\mathcal{L}_0$ is the homogenized operator associated with
$\partial_t + \mathcal{L}_\varep$.
The proof of the following theorem may be found in \cite[pp.140-142]{bensoussan-1978}.

\begin{theorem}\label{homo-theorem-1}
Let $\Omega$ be a bounded Lipschitz domain in $\mathbb{R}^d$. Let
$u_\varep\in L^2(0,T; W^{1,2}_0(\Omega))\cap L^\infty(0,T; L^2(\Omega))$ be a weak solution of
\begin{equation}\label{3.1}
\big( \partial_t 
+\mathcal{L}_\varep\big)  u_\varep =F \quad \text{ in } \Omega\times (0,T) \quad \text{ and } \quad 
u_\varep =h \quad \text{ on } \Omega \times \{ 0\},
\end{equation}
where $F\in L^2(0, T; W^{-1, 2}(\Omega))$ and $h \in L^2(\Omega)$.
Then, as $\varep\to 0$,
$$
u_\varep \to u_0 \text{ weakly in } L^2(0, T; W^{1,2}_0(\Omega)) \text{ and strongly in } L^2(0,T; L^2(\Omega)),
$$
where $u_0$ is the unique weak solution in $L^2(0,T; W^{1,2}_0(\Omega))\cap L^\infty(0,T; L^2(\Omega))$ of 
\begin{equation}\label{3.2}
\big( \partial_t 
+\mathcal{L}_0\big) u_0 =F \quad \text{ in } \Omega\times (0,T) \quad \text{ and } \quad 
u_0 =h \quad \text{ on } \Omega \times \{ 0\}.
\end{equation}

\end{theorem}

\begin{remark}\label{remark-3.1}
{\rm
Using (\ref{cell-problem}), we may  write
\begin{equation}\label{hat-formula-1}
\aligned
\widehat{a}^{\alpha\beta}_{ij}
&= \int_Y 
A\nabla \left\{ P_j^\beta +\chi_j^\beta\right\}\cdot \nabla \big\{ P_i^\alpha +\chi_i^\alpha\big\} \\
& \qquad \qquad \qquad 
+\int_0^1 <\frac{\partial }{\partial s} \left (\chi_j^{\beta}\right), \chi_i^{\alpha} >_{[0,1]^d} \, ds.
\endaligned
\end{equation}
It follows that for any $\xi=(\xi_i^\alpha)\in \mathbb{R}^{dm}$,
\begin{equation}\label{hat-formula-2}
\widehat{a}_{ij}^{\alpha\beta}\xi_i^\alpha\xi_j^\beta
=
\int_Y A\nabla \big\{ \xi_j^\beta P_j^\beta +\xi_j^\beta\chi_j^\beta\big\} \cdot \nabla 
\big\{ \xi_i^\alpha P_i^\alpha +\xi_i^\alpha \chi_i^\alpha \big\},
\end{equation}
where we have used
$$
\int_0^1 <\frac{\partial}{\partial s}
\big( \xi_j^\beta \chi_j^{\beta}\big), \xi_i^\alpha \chi_i^{\alpha} >_{[0,1]^d}\, ds =0.
$$
Using the ellipticity condition (\ref{e1.3}), we may deduce from (\ref{hat-formula-2}) and integration by parts  that
$$
\aligned
\widehat{a}_{ij}^{\alpha\beta} \xi_i^\alpha\xi_j^\beta
& \ge
\mu \int_Y 
\nabla \big\{ \xi_j^\beta P_j^\beta +\xi_j^\beta\chi_j^\beta\big\} \cdot \nabla 
\big\{ \xi_i^\alpha P_i^\alpha +\xi_i^\alpha \chi_i^\alpha \big\}\\
&= \mu|\xi|^2
+\mu
\int_Y 
\nabla \big\{ \xi_j^\beta\chi_j^\beta\big\} \cdot \nabla \big\{ \xi_i^\alpha \chi_i^\alpha\big\} \\
&\ge \mu |\xi|^2.
\endaligned
$$
Hence, for any $\xi=(\xi_i^\alpha)\in \mathbb{R}^{dm}$,
\begin{equation}\label{ellipticity-1}
\mu |\xi|^2 \le \widehat{a}_{ij}^{\alpha\beta} \xi_i^\alpha \xi_j^\beta \le \mu_1 |\xi|^2,
\end{equation}
where $\mu_1$ depends only on $d$, $m$, and $\mu$.
This gives the ellipticity of the homogenized matrix $\widehat{A}=\big( \widehat{a}_{ij}^{\alpha\beta}\big)$.
}
\end{remark}

\begin{remark}\label{remark-3.2}
{\rm
Let $\chi^* (y,s)=(\chi_i^{*\alpha} (y,s))$ denote the matrix of correctors for 
$$
-\partial_t -\text{div}
\big(A^* (x/\varep, t/\varep^2)\nabla\big),
$$
 where $A^* (y,s)=(b_{ij}^{\alpha\beta} (y,s))$ with $b_{ij}^{\alpha\beta}
=a_{ji}^{\beta\alpha}$ is the adjoint of $A=\big(a_{ij}^{\alpha\beta}\big)$. That is,
\begin{equation}\label{adjoint-corrector}
\left\{
\aligned
& 
-\frac{\partial}{\partial s} \big( \chi_i^{*\alpha}\big)
-\text{\rm div} \left(A^* \nabla \chi_i^{*\alpha}\right)
=\text{\rm div} \left( A^* \nabla P_i^\alpha\right) \quad \text{ in } \mathbb{R}^{d+1},\\
& \chi_i^{*\alpha} \text{ is $Y$-periodic},\\
&\int_Y \chi_i^{*\alpha} =0.
\endaligned
\right.
\end{equation}
It follows from (\ref{adjoint-corrector}) that
\begin{equation}\label{adjoint-3.1}
-\int_0^1 <\frac{\partial}{\partial s} \big(\chi_i^{*\alpha}\big), \chi_j^\beta>_{[0,1]^d}\, ds+\int_Y A^*\nabla \chi_i^{*\alpha} \cdot \nabla \chi_j^\beta
=\int_Y A^*\nabla P_i^\alpha \cdot \nabla \chi_j^\beta.
\end{equation}
Similarly, by (\ref{cell-problem}),
\begin{equation}\label{adjoint-3.2}
\int_0^1 <\frac{\partial }{\partial s}\big(\chi_j^\beta\big), \chi_i^{*\alpha}>_{[0,1]^d}\, ds
+\int_Y A\nabla \chi_j^\beta \cdot \nabla \chi_i^{*\alpha}
=\int_Y A\nabla P_j^\beta \cdot \nabla \chi_i^{*\alpha}.
\end{equation}
In view of (\ref{adjoint-3.1})-(\ref{adjoint-3.2}) we obtain
\begin{equation}\label{adjoint-3.3}
\aligned
\int_Y A\nabla \chi_j^\beta\cdot \nabla P_i^\alpha &=
\int_Y A^* \nabla P_i^\alpha \cdot \nabla \chi_j^\beta\\
& =\int_Y A\nabla P_j^\beta \cdot \nabla \chi_i^{*\alpha}
=\int_Y A^*\nabla\chi_i^{*\alpha} \cdot \nabla P_j^\beta.
\endaligned
\end{equation}
This, together with (\ref{hat-formula-1}), gives another formula for $\widehat{a}_{ij}^{\alpha\beta}$:
\begin{equation}\label{hat-formula-3}
\widehat{a}_{ij}^{\alpha\beta}
=\int_Y
A^* \nabla \big( \chi_i^{*\alpha}+P_i^\alpha\big) \cdot \nabla P_j^\beta.
\end{equation}
}
\end{remark}

A compactness argument will be used in following sections to establish uniform interior and boundary estimates.
This would require us to consider a sequence of matrices $\{ A_k(y,s)\} $ satisfying conditions (\ref{e1.3})-(\ref{e1.4}).
An extension of Theorem \ref{homo-theorem-1} is thus needed,

\begin{lemma}\label{lemma-3.1}
Let $h_k=h_k(y,s)$ be a sequence of locally square integrable and $Y$-periodic functions in $\mathbb{R}^{d+1}$.
Suppose that  $\| h_k\|_{L^2(Y)} \le C$ and as $k\to \infty$,
$
\int_Y h_k \to M.
$
Then, for any $\varep_k \to 0$,
$$
h_k (x/\varep_k, t/\varep_k^2) \to M \quad \text{ weakly in } L^2 (Q)
$$
for any bounded domain $Q$ in $\mathbb{R}^{d+1}$.
\end{lemma}

\begin{proof}
The proof is standard.
By considering $h_k -\int_Y h_k$ we may assume that $\int_Y h_k =0$ and hence, $M=0$.
Using $\| h_k\|_{L^2(Y)} \le C$, one may show that the sequence
$\{ h_k (x/\varep_k, t/\varep_k^2 )\}$ is bounded in $L^2(Q)$.
Thus, it suffices to show that
$$
\int_Q h_k (x/\varep_k, t/\varep_k^2)\, \varphi (x,t)\, dxdt \to 0, \text{ as } k\to \infty,
$$
for any $\varphi \in C_0^1 (Q)$.

To this end let $u_k\in W^{2,2}_{loc} (\mathbb{R}^{d+1}) $ be an $Y$-periodic function in $\mathbb{R}^{d+1}$ such that
$\Delta_{(y,s)} u_k =h_k$ in $\mathbb{R}^{d+1}$, where $\Delta_{(y,s)}$ denotes the Laplacian in 
$\mathbb{R}^{d+1}$.
It follows from integration by parts that
\begin{equation}\label{3.0}
\aligned
&\int_Q h_k (x/\varep_k, t/\varep_k^2) \varphi (x,t)\, dxdt\\
&=-\varep_k \int_Q \nabla u_k (x/\varep_k, t/\varep_k^2)\cdot \nabla \varphi (x,t)\, dxdt
-\varep_k^2 \int_Q \frac{\partial u_k}{\partial t} (x/\varep_k, t/\varep_k^2) \frac{\partial\varphi}{\partial t}\, dxdt.
\endaligned
\end{equation}
Since $\|\nabla u_k\|_{L^2(Y)} +\| \partial_t  u_k \|_{L^2(Y)} \le C \| h_k\|_{L^2(Y)}\le C$,
the sequences $\{ \nabla u_k (x/\varep_k, t/\varep_k^2)\}$ and $\{ \partial_t  u_k (x/\varep_k, t/\varep_k^2)\}$
are bounded in $L^2(Q)$.
The desired result follows easily from (\ref{3.0}) by the Cauchy inequality.
\end{proof}

The next lemma is due to  J.P. Aubin and J. L. Lions.

\begin{lemma}\label{embedding-lemma}
Let $X_0\subset X\subset X_1$ be three Banach spaces.
Suppose that $X_0, X_1$ are reflexive and that the injection $X_0 \subset X$ is compact.
Let $1<\alpha_0,\alpha_1<\infty$. Define
$$
Y=\left\{ u: \ u \in L^{\alpha_0} (T_0, T_1; X_0) \text{ and } \partial_t  u \in L^{\alpha_1} (T_0, T_1; X_1) \right\}
$$
with norm 
$$
\| u\|_Y =\| u\|_{L^{\alpha_0} (T_0, T_1; X_0)} +\| \partial_t u\|_{L^{\alpha_1} (T_0, T_1; X_1)}.
$$
Then $Y$ is a Banach space, and the injection 
$Y\subset L^{\alpha_0} (T_0, T_1; X)$ is compact.
\end{lemma}

\begin{theorem}\label{homo-theorem-2}
Let $\Omega$ be a bounded Lipschitz domain in $\mathbb{R}^d$ and $F \in L^2 (T_0, T_1; W^{-1,2} (\Omega))$.
Let $u_k \in L^2(T_0,T_1; W^{1,2}(\Omega))$ be a weak solution of
\begin{equation}\label{equation-k}
\partial_t u_k -\text{\rm div} \big[ A_k (x/\varep_k, t/\varep_k^2) \nabla u_k\big]
= F \quad \text{ in } \Omega \times (T_0, T_1),
\end{equation}
where $\varep_k \to 0$ and the matrix $A_k(y,s)$ satisfies (\ref{e1.3})-(\ref{e1.4}).
Suppose that $\widehat{A_k} \to A^0$, and 
$$
\left\{
\aligned
u_k  & \to u &\quad& \text{ weakly in } L^2(T_0, T_1; L^2(\Omega)),\\
\nabla u_k  & \to \nabla u& \quad & \text{ weakly in } L^2(T_0, T_1; L^2 (\Omega)).\\
\endaligned
\right.
$$
Then $u\in L^2(0, T; W^{1,2}(\Omega))$ is a  solution of
$$
\partial_t u -\text{\rm div} \big( A^0 \nabla u\big)=F \quad \text{ in } \Omega \times (T_0, T_1),
$$
and the constant  matrix $A^0$ satisfies the ellipticity condition (\ref{ellipticity-1}).
\end{theorem}

\begin{proof}
The proof is similar to that of Theorem 2.1 in \cite[p.140]{bensoussan-1978} for $k=2$.
We provide a proof here for the sake of completeness.

We first note that since $\{ u_k\}$ is bounded in $L^2(T_0, T_1; W^{1,2}(\Omega))$, by (\ref{equation-k}),
 $\{ \partial_t u_k\}$ is bounded in $L^2(T_0, T_1; W^{-1,2}(\Omega))$. In view of Lemma \ref{embedding-lemma},
by passing to a subsequence, we may assume that $u_k\to u$
strongly in $L^2(T_0, T_1; L^2(\Omega))$. Since $\{ A_k (x/\varep_k, t/\varep_k)\nabla u_k\}$ is bounded in 
$L^2 (T_0, T_1; L^2(\Omega))$,
by passing to another subsequence, we may further assume that 
$
\{ A_k (x/\varep_k, t/\varep_k^2)\nabla u_k \}
 \text{ converges weakly in } L^2(T_0, T_1; L^2(\Omega)).
$

Let $P$ be an $\mathbb{R}^m$-valued affine function in $y$ and $\omega_{k}$
be the (weak) solution of the following cell problem:
\begin{equation}\label{cell-problem-1}
\left\{
\aligned
& -\frac{\partial \omega_k}{\partial s}
-\text{div} \big( A_k^* \nabla \omega_k\big)  =\text{div} \big( A_k^* \nabla P\big) \quad \text{ in } \mathbb{R}^{d+1},\\
& \omega_k =\omega_k (y,s) \text{ is $Y$-periodic},\\
& \int_Y \omega_k  =0.
\endaligned
\right.
\end{equation}
Let
$$
\theta_k=\theta_k (x,t)
=\varep_k \omega_k (x/\varep_k, t/\varep^2_k) +P (x).
$$
Note that
$$
\left(-\partial_t +\mathcal{L}_{\varep_k}^{k*} \right) \theta_k =0 \quad \text{ in } \mathbb{R}^{d+1},
$$
where $\mathcal{L}_\varep^{k*} =-\text{div} (A^*_k (x/\varep, t/\varep^2)\nabla)$,
and for any scalar function $\phi$ in $C_0^1 (\Omega\times (T_0, T_1))$,
\begin{equation}\label{3.5}
\left\{
\aligned
\int_{T_0}^{T_1} <\frac{\partial u_k}{\partial t}, \phi \theta_k> dt
+\int_{T_0}^{T_1} \int_\Omega A_k (x/\varep_k, t/\varep_k^2)\nabla u_k \cdot \nabla (\phi\theta_k )\, dx dt
&=\int_{T_0}^{T_1} < F, \phi\theta_k > dt,\\
-\int_{T_0}^{T_1}<\frac{\partial\theta_k}{\partial t}, \phi u_k > dt
+\int_{T_0}^{T_1}\int_\Omega A_k^* (x/\varep_k, t/\varep_k^2)\nabla \theta_k \cdot
\nabla (\phi u_k)\, dxdt & =0.
\endaligned
\right.
\end{equation}
It follows from (\ref{3.5}) that
\begin{equation}\label{3.6}
\aligned
& \int_{T_0}^{T_1}\int_\Omega
\left\{ A_k (x/\varep_k, t/\varep_k^2) \nabla u_k \cdot \theta_k
-A_k^* (x/\varep_k, t/\varep_k^2) \nabla \theta_k \cdot u_k \right\} \nabla \phi\, dxdt\\
& =
\int_{T_0}^{T_1} <u_k, \theta_k \frac{\partial\phi}{\partial t} > dt + \int_{T_0}^{T_1} <F, \phi \theta_k> dt.
\endaligned
\end{equation}
It is easy to see that $\theta_k \to P$ strongly in $L^2 (T_0,T_1; L^2(\Omega))$. 
Since $u_k\to u$ weakly in $L^2(T_0,T_1;L^2(\Omega))$, this implies that the right hand side of (\ref{3.6}) converges to
$$
\int_{T_0}^{T_1} <u, P\frac{\partial \phi}{\partial t}>  dt+\int_{T_0}^{T_1} <F, \phi  P> dt.
$$

Let $\eta=\eta (x,t)$ denote the weak limit of
$
A_k(x/\varep_k, t/\varep_k^2) \nabla u_k  \text{ in }L^2(T_0, T_1; L^2(\Omega))
$. Then
\begin{equation}\label{3.0-0-0}
-\int_{T_0}^{T_1}<u, \frac{\partial \psi}{\partial t}> dt
+\int_{T_0}^{T_1} <\eta, \nabla \psi> dt =\int_{T_0}^{T_1} <F, \psi> dt,
\end{equation}
for any $\mathbb{R}^m$-valued function $\psi$ in $C_0^1(\Omega\times (T_0, T_1))$. In particular,
taking $\psi =\phi P$, we obtain 
$$
\int_{T_0}^{T_1} <\eta, \nabla (\phi P)> dt =
\int_{T_0}^{T_1} < F, \phi P> dt +\int_{T_0}^{T_1} < u, P\frac{\partial \phi}{\partial t}> dt.
$$
As a result, the right hand side of (\ref{3.6}) converges to
\begin{equation}\label{3.0-0}
\int_{T_0}^{T_1} <\eta, \nabla (\phi P)>  dt= \int_{T_0}^{T_1}<\eta, P \nabla \phi > +
\int_{T_0}^{T_1} <\eta, \phi\nabla P> dt.
\end{equation}

Next, using the fact that $\theta_k \to P$ strongly in $L^2(T_0, T_1; L^2(\Omega))$, we see that
$$
\int_{T_0}^{T_1} \int_\Omega
A_k (x/\varep_k, t/\varep_k^2)\nabla u_k \cdot \theta_k \nabla \phi\, dxdt
\to \int_{T_0}^{T_1} <\eta, P\nabla \phi> dt.
$$
In view of (\ref{3.6}) and (\ref{3.0-0}) we have proved that
\begin{equation}\label{3.0-1}
-\int_{T_0}^{T_1} \int_\Omega A_k^* (x/\varep_k, t/\varep_k^2) \nabla \theta_k \cdot u_k \nabla \phi\, dxdt
\to \int_{T_0}^{T_1} <\eta, \phi \nabla P> dt.
\end{equation}
On the other hand, by Lemma \ref{lemma-3.1}, 
\begin{equation}\label{3.0-2}
A_k^* (x/\varep_k, t/\varep_k^2)\nabla \theta_k
\to M=\lim_{k\to \infty}
\int_Y A_k^* (y,s) \left\{ \nabla w_k (y,s) +\nabla P \right\}\, dyds
\end{equation}
 weakly in $ L^2(\Omega \times (T_0, T_1))$, provided the limit in the right hand side of (\ref{3.0-2}) exists.
 It follows that the left hand side of (\ref{3.0-1}) also converges to
 $$
 -M\int_{T_0}^{T_1} \int_\Omega u \nabla \phi\, dxdt
 =M \int_{T_0}^{T_1} \int_\Omega (\nabla u) \phi\, dxdt.
 $$
 Consequently, we obtain
 $$
 M \int_{T_0}^{T_1} \int_\Omega (\nabla u) \phi\, dxdt = \int_{T_0}^{T_1} <\eta, \phi \nabla P> dt.
 $$
 Since $\phi\in C_0^1(\Omega\times (T_0, T_1))$ is arbitrary, this gives $\eta \nabla P =M \nabla u$.
 
 Finally, note that if we take $P=P_i^\alpha =y_i (0, \dots,1, \dots, 0)$
 with $1$ in the $\alpha^{th}$ position, then
 $\omega_k=(\omega_k^\gamma) =(\chi_{i,k}^{*\gamma\alpha})$ are the correctors for the
 operator $-\partial_t  -\text{div} \big(A^*_k (x/\varep, t/\varep^2)\nabla \big)$.
 Let
 $$
 C_{ij,k}^{\alpha\beta}
  =\iint_Y
  a_{j\ell,k}^{* \beta\gamma}
\frac{\partial}{\partial y_\ell} \big\{ \omega_k^\gamma  +P\big\}
=\iint_Y A^*_k \nabla \big( \chi_{i,k}^{*\alpha}+P_j^\alpha\big)\cdot \nabla P_j^\beta.
 $$
 It follows from (\ref{hat-formula-3}) that 
 $(C_{ij,k}^{\alpha\beta})=\widehat{A_k}$ is the homogenized matrix for $A_k$ and thus satisfies the ellipticity condition (\ref{ellipticity-1}).
Since $\widehat{A_k}\to A^0$,  we see that
$\eta =A^0 \nabla u$ and $A^0$ satisfies (\ref{ellipticity-1}). This, together with (\ref{3.0-0-0}), gives
$$
\partial_t u -\text{\rm div} \big( A^0 \nabla u\big)=F \quad \text{ in } \Omega \times (T_0, T_1),
$$
and completes the proof.
\end{proof}


\section{Interior H\"older estimates}
\setcounter{equation}{0}

In this section, as an intermediate step,
 we establish the interior H\"older estimates for local solutions of
\begin{equation}\label{4.1}
\partial_t u_\varep-\text{div} \big(A(x/\varep,t/\varep^2)\nabla u_\varep \big)=0,
\end{equation}
under the conditions (\ref{e1.3})-(\ref{e1.4}) and $A\in VMO_x$.

\begin{theorem}\label{theorem-4.1}
Suppose that $A=A(y,s)$ satisfies (\ref{e1.3})-(\ref{e1.4}).
Also assume that $A\in VMO_x$.
Let $u_\varep$ be a weak solution of  (\ref{4.1}) in $2Q$ for
some $Q=Q_R (x_0, t_0)$.
Then, for any $0<\alpha<1$,
\begin{equation}\label{Holder-4.1}
\| u_\varep\|_{C^{\alpha, \alpha/2} (Q)}
\le 
\frac{C}{R^\alpha} \left\{ \average_{2Q} |u_\varep|^2\right\}^{1/2},
\end{equation}
where $C$ depends at most on $d$, $m$, $\alpha$, $\mu$, and $\omega(r)$ in (\ref{VMO-x}).
In particular,
\begin{equation}\label{L-infty}
\|u_\varep\|_{L^\infty (Q)}
\le C
\left\{ \average_{2Q} |u_\varep|^2 \right\}^{1/2}.
\end{equation}
\end{theorem}

By translation and rescaling property (\ref{rescaling}), to prove (\ref{Holder-4.1}),  we may assume that
$Q=Q_1 (0, 0)$.
In view of Lemma \ref{Cam-Lemma}, it suffices to prove that
\begin{equation}\label{4.2}
\average_{Q_r (x,t)}
|u_\varep - \average_{Q_r(x,t)} u_\varep |^2
\le C r^{2\alpha}
\average_{2Q} |u_\varep|^2,
\end{equation}
for any $(x,t)\in Q$ and $0<r<1/2$.
Furthermore, by translation and a simple covering argument, we only need to show that
if $u_\varep$ is a weak solution of (\ref{4.1}) in $Q_1$,
\begin{equation}\label{4.3}
\average_{Q_r} |u_\varep -\average_{Q_r} u_\varep |^2
\le C r^{2\alpha} \average_{Q_1} |u_\varep|^2,
\end{equation}
for any $0<r<1/2$, where $Q_r =Q_r (0,0)$.
This will be done by a three-step compactness argument,
similar to that used in \cite{AL-1987}.

\begin{lemma}\label{lemma-4.1}
Let $0<\alpha<1$.
Then there exist constants $\varep_0>0$ and $\theta\in (0, 1/4)$, depending only on $d$, $m$, $\alpha$ and $\mu$, such that
\begin{equation}\label{4.1-1}
\average_{Q_\theta} |u_\varep
-\average_{Q_\theta} u_\varep  |^2 \le \theta^{2\alpha},
\end{equation}
whenever $0<\varep<\varep_0$, $u_\varep$ is a weak solution of (\ref{4.1}) in $Q_1$, and
$$
\average_{Q_1} |u_\varep|^2 \le 1.
$$
\end{lemma}

\begin{proof}
The lemma is proved by contradiction, using Theorem \ref{homo-theorem-2} and the
fact that for any $\theta\in (0,1/4)$,
\begin{equation}\label{4.1-2}
\average_{Q_\theta} 
|u - \average_{Q_\theta} u |^2
\le {C_0 \theta^2}
\average_{Q_{1/2}}
|u|^2,
\end{equation}
where $\partial_t u-\text{div} \big( A^0\nabla u\big) =0$ in $Q_{1/2}$
and $A^0$ is a constant matrix satisfying the ellipticity condition (\ref{ellipticity-1}).
The constant $C_0$ in (\ref{4.1-2}) depends only on $d$, $m$, and $\mu$.

Fix $0<\alpha<1$. Choose $\theta\in (0,1/4)$ so small that $2^{d+2} C_0\theta^2< \theta^{2\alpha}$.
We claim that the estimate (\ref{4.1-1}) holds for this $\theta$ and some $\varep_0>0$,
which depends only on $d$, $m$, and $\mu$.

Suppose this is not the case. Then there exist sequences $\{ \varep_k\}$, $\{ A_k\}$
satisfying (\ref{e1.3})-(\ref{e1.4}), and $\{ u_k\}\subset L^2 (-1,0; W^{1,2} (B(0,1)))$ such that $\varep_k\to 0$,
\begin{equation}
 \partial_t u_k -\text{div} \big( A_k (x/\varep_k, t/\varep_k^2)\nabla u_k \big) =0 \qquad \text{ in } Q_1,
 \end{equation}
and
\begin{equation}\label{4.1-3}
 \average_{Q_1} |u_k|^2\le 1
 \quad \text{ and } 
 \quad
\average_{Q_\theta}
|u_k - \average_{Q_\theta} u_k |^2 > \theta^{2\alpha}.
\end{equation}
Since $\{ u_k\}$ is bounded in $L^2(Q_1)$, it follows from the energy estimate (\ref{Caccio}) that
$\{ \nabla u_k\}$ is bounded in $L^2(Q_{1/2})$. 
Hence, by passing to subsequences, we may assume that
$$
\left\{
\aligned
u_k  & \to u &\quad &\text{ weakly in } L^2(Q_1),\\
\nabla u_k & \to \nabla u &\quad& \text{ weakly in } L^2(Q_{1/2}).
\endaligned
\right.
$$
Also, note that $\{ \partial_t u_k\}$ is bounded in $L^2(-1/4, 0; W^{-1,2}(B(0,1/2))$.
In view of Lemma \ref{embedding-lemma} we may assume that
$u_k \to u$ strongly in $L^2(Q_{1/2})$.
It then follows from (\ref{4.1-3}) that
\begin{equation}\label{4.1-4}
\average_{Q_1} |u|^2\le 1
\text{ and }
\average_{Q_\theta}
|u - \average_{Q_\theta} u |^2  \ge \theta^{2\alpha}.
\end{equation}

Finally, by Theorem \ref{homo-theorem-2}, the function $u$ is a solution of $\partial_t u -\text{div} (A^0 \nabla u)=0$
in $Q_{1/2}$ for some constant matrix $A^0$ satisfying (\ref{ellipticity-1});
as a result, the estimate (\ref{4.1-2}) holds.
This, together with (\ref{4.1-4}), gives
$$
\theta^{2\alpha}
\le {C_0 \theta^2}
\average_{Q_{1/2}} |u|^2
\le 
 {2^{d+2} C_0 \theta^2}
\average_{Q_{1}} |u|^2
\le 2^{d+2} C_0 \theta^2,
$$
which is in contradiction with the choice of $\theta$.
\end{proof}

\begin{lemma}\label{lemma-4.2}
Fix $0<\alpha<1$.
Let $\varep_0$ and $\theta$ be given by Lemma \ref{lemma-4.1}.
Suppose that $u_\varep$ is a weak solution of
$\big (\partial_t +\mathcal{L}_\varep\big) u_\varep=0$
in $Q_1$.
Then, if $0<\varep< \varep_0 \theta^{k-1}$ for some $k\ge 1$,
\begin{equation}\label{4.2-1}
\average_{Q_{\theta^k}}
|u_\varep - \average_{Q_{\theta^k}} u_\varep |^2
\le 
\theta^{2k \alpha}
\average_{Q_1} |u_\varep|^2.
\end{equation}
\end{lemma}

\begin{proof}
The lemma is proved by an induction argument on $k$.
The case $k=1$ follows directly from Lemma \ref{lemma-4.1}.

Suppose that the estimate (\ref{4.2-1}) holds for some $k\ge 1$.
Let $\big( \partial_t  +\mathcal{L}_\varep \big) u_\varep =0$ in $Q_1$
and $0<\varep<\varep_0 \theta^k$.
Define
$$
w(x,t)
=\theta^{-\alpha k} \left\{ u_\varep (\theta^k x, \theta^{2k} t)
-\average_{Q_{\theta^k}} u_\varep \right\}/
\left\{ \average_{Q_1} |u_\varep|^2 \right\}^{1/2}.
$$
Then 
$$
\left( \partial_t +\mathcal{L}_{\frac{\varep}{\theta^k}}\right) w =0 \quad \text{ in } Q_1,
$$
and by the induction assumption,
$$
\average_{Q_1} |w|^2\le 1.
$$
Since $\varep/\theta^k <\varep_0$, by Lemma \ref{lemma-4.1}, we obtain
$$
\average_{Q_\theta} |w -\average_{Q_\theta} w|^2 \le \theta^{2\alpha},
$$
which leads to
$$
\average_{Q_{\theta^{k+1}}} |u_\varep
- \average_{Q_{\theta^{k+1}}} u_\varep|^2
\le 
\theta^{2(k+1)\alpha} \average_{Q_1} |u_\varep|^2.
$$
This completes the proof.
\end{proof}

\begin{proof}[\bf Proof of Theorem \ref{theorem-4.1}]

Recall that it suffices to prove
\begin{equation}\label{4.3-1}
\average_{Q_r}
|u_\varep -\average_{Q_r} u_\varep|^2 
\le C r^{2\alpha} \average_{Q_1} |u_\varep|^2
\end{equation}
for $0<r<1/2$, where $u_\varep$ is a weak solution of $\big(\partial_t +\mathcal{L}_\varep\big) u_\varep =0$ in $Q_1$.
To this end we first point out that if $\varep\ge \theta\varep_0$, the estimate (\ref{4.3-1})
follows directly from the  regularity theory for parabolic systems with $VMO_x$ coefficients in \cite{Byun-2007, Krylov-2007}.
See (\ref{krylov-1}) and Remark \ref{remark-2.1}.

Suppose now that $0<\varep< \theta\varep_0$.
Consider the case $\varep/\varep_0 \le r<\theta$.
Choose $k\ge 1$ such that $\theta^{k+1} \le r< \theta^k$.
Since $\theta^k >\varep/\varep_0$, by Lemma \ref{lemma-4.2},
$$
\aligned
\average_{Q_r} | u_\varep -\average_{Q_r} u_\varep|^2
& \le C
\average_{Q_{\theta^k}} | u_\varep -\average_{Q_{\theta^k}} u_\varep|^2\\
& \le C \theta^{2k\alpha}\average_{Q_1} |u_\varep|^2
\le C r^{2\alpha}
\average_{Q_1} |u_\varep|^2.
\endaligned
$$

Finally, we need to handle the case $0<r<\varep/\varep_0$ (the case $\theta<r<1/2$ is trivial).
We use a blow-up argument.
Let 
$$
w(x,t)=u_\varep (\varep x, \varep^2 t)-\average_{Q_{\frac{2\varep}{\varep_0}}} u_\varep.
$$
Note that $\big(\partial_t +\mathcal{L}_1\big) w =0$ in $Q_{2/\varep_0}$.
By the H\"older estimate (\ref{krylov-1})
 for second-order parabolic systems with $VMO_x$ coefficients, we have
$$
\average_{Q_\rho} |w -\average_{Q_\rho} w|^2
\le C \rho^{2\alpha} \average_{Q_{2/\varep_0}} |w|^2.
$$
Hence, if $0<r<\varep/\varep_0$,
$$
\aligned
\average_{Q_r}
|u_\varep -\average_{Q_r} u_\varep|^2
& \le C \left( \frac{r}{\varep}\right)^{2\alpha}
\average_{Q_{2\varep/\varep_0}}
|u_\varep -\average_{Q_{2\varep/\varep_0}} u_\varep|^2\\
& \le C r^{2\alpha}
\average_{Q_1} |u_\varep|^2,
\endaligned
$$
where we have used (\ref{4.3-1}) for $r=2\varep/\varep_0$ in the last inequality.
This completes the proof of Theorem \ref{theorem-4.1}.
\end{proof}


\section{Interior Lipschitz estimates}
\setcounter{equation}{0}

In this section we establish the interior Lipschitz estimates.
This requires some stronger smoothness condition on $A$.
We shall call $A\in \Lambda (\mu, \lambda, \tau)$ if $A$ satisfies the ellipticity and periodicity conditions
(\ref{e1.3})-(\ref{e1.4}) and the H\"older continuity condition,
\begin{equation}\label{Holder-condition}
|A(x,t)-A(y,s)|\le \tau |(x,t)- (y,s)|^\lambda, \quad \text{ for any } (x,t), (y,s)\in \mathbb{R}^{d+1}.
\end{equation}

\begin{theorem}\label{theorem-5.1}
Suppose that $A\in \Lambda(\mu, \lambda, \tau)$.
Let $u_\varep$ be a weak solution of $\big(\partial_t +\mathcal{L}_\varep\big) u_\varep =0$ in $2Q$
for some $Q=Q_r (x_0, t_0)$.
Then
\begin{equation}\label{5.1}
\| u_\varep\|_{C^{1,1/2} (Q)}
\le \frac{C}{r} \left\{ \average_{2Q} |u_\varep|^2 \right\}^{1/2},
\end{equation}
where $C$ depends  at most on $d$, $m$, $\mu$, $\lambda$, and $\tau$.
\end{theorem}

As in the case of H\"older estimates, Theorem \ref{theorem-5.1} is also proved by a three-step
compactness argument. Our proof follows closely  the elliptic case in \cite{AL-1987}.

Recall that $Q_r=Q_r (0,0)=B(0,r)\times (-r^2, 0)$, $P_\ell^\beta (x)=x_\ell (0, \dots, 1, \dots)$
with $1$ in the $\beta^{th}$ position, and $(\chi_\ell^\beta)$ are the correctors defined by (\ref{cell-problem}).

\begin{lemma}\label{lemma-5.1}
There exist constants $\varep_0$ and $\theta\in (0,1/4)$,
depending only on $d$, $m$, $\mu$, $\lambda$, and $\tau$, such that for
$0<\varep<\varep_0$,
\begin{equation}\label{5.1-1}
\aligned
\sup_{(x,t)\in Q_\theta}
\big|
u_\varep (x, t)-u_\varep(0,0)
-& \left\{ P_\ell^\beta (x) +\varep \chi^\beta_\ell (x/\varep, t/\varep^2)\right\} \average_{Q_\theta}
\frac{\partial u_\varep^\beta}{\partial x_\ell} \big|\\
& \le \theta^{3/2} 
\| u_\varep\|_{L^\infty(Q_1)},
\endaligned
\end{equation}
whenever $u_\varep$ is a weak solution of $\big(\partial_t +\mathcal{L}_\varep \big) u_\varep =0$ in $Q_2$.
\end{lemma}

\begin{proof}
As in the case of Lemma \ref{lemma-4.1},
Lemma \ref{lemma-5.1}  is proved by contradiction, using Theorem \ref{homo-theorem-2} as well as the fact that
for any $\theta\in (0,1/4)$,
\begin{equation}\label{5.1-2}
\sup_{(x,t)\in Q_\theta}
\big| u(x,t)-u(0,0)-x_\ell \average_{Q_\theta} \frac{\partial u}{\partial x_\ell} \big|
\le C_0 \theta^2 \left\{ \average_{Q_{1/2}} |u|^2\right\}^{1/2},
\end{equation}
where $u$ is a  solution of $\partial_t u -\text{div}(A^0\nabla u) =0$ in $Q_{1/2}$,
and the constant matrix $A^0$ satisfies the ellipticity condition
(\ref{ellipticity-1}). The estimate (\ref{5.1-2}) follows easily from the standard regularity estimate,
$$
\| \nabla^2 u\|_{L^\infty (Q_{1/4})} +\| \partial_t u\|_{L^\infty(Q_{1/4})}
\le C \| u\|_{L^2(Q_{1/2})},
$$
for solutions of second-order parabolic systems with constant coefficients. 
The constant $C_0$ in (\ref{5.1-2}) depends only on $d$, $m$, and $\mu$.

Choose $\theta\in (0,1/4)$ so small that $2^{d} C_0\theta^2<\theta^{3/2}$.
We claim that the estimate (\ref{5.1-1}) holds for this $\theta$ and some $\varep_0>0$,
which depends only on $d$, $m$, $\mu$, $\lambda$, and $\tau$.

Suppose this is not the case. Then there exist sequences $\{\varep_k\}$, 
$\{A_k\}\subset \Lambda (\mu, \lambda, \tau)$, and
$\{ u_k\}$ such that $\varep_k\to 0$,
\begin{equation}\label{5.1-3}
\left\{
\aligned
& \partial_t u_k -\text{div} \big(A_k (x/\varep_k, t/\varep_k^2)\nabla u_k \big) =0 \quad \text{ in } Q_1,\\
& \| u_k\|_{L^\infty (Q_1)}\le 1,
\endaligned
\right.
\end{equation}
and
\begin{equation}\label{5.1-4}
\sup_{(x,t)\in Q_\theta}
\big| u_k (x,t)-u_k (0,0)- 
 \left\{ P_\ell ^\beta (x)
+\varep_k \chi_{\ell}^{k, \beta} (x/\varep_k, t/\varep_k^2) \right\}
\average_{Q_\theta} \frac{\partial u_k^\beta}{\partial x_\ell} \big|
 >\theta^{3/2},
\end{equation}
where $\chi_\ell^{k, \beta}$ are the correctors associated with the periodic matrix $A_k$.
By passing to subsequences, as in the proof of Lemma \ref{lemma-4.1},
we may assume that $u_k \to u$ weakly in $L^2(Q_1)$ and
$\nabla u_k \to \nabla u$ weakly in $L^2(Q_{1/2})$.
Observe that by Theorem \ref{homo-theorem-2}, the function $u$ 
is a solution of $\partial_t u -\text{div}(A^0\nabla u)=0$ in $Q_{1/2}$ for some
constant matrix $A^0$ satisfying (\ref{ellipticity-1}).
Consequently, it satisfies the estimate (\ref{5.1-2}).

Finally, we note that by Theorem \ref{theorem-4.1}, the sequence $\{ u_k \}$
is bounded in $C^{\alpha, \alpha/2}(Q_{1/4})$ for any $\alpha \in (0,1)$.
Thus, by passing to a subsequence, we may assume that
$u_k \to u$ uniformly on $Q_{1/4}$.
This allows us to take the limit in $k$ in (\ref{5.1-4}).
Indeed, since $\{ \chi_\ell^{k, \beta}\}$ is bounded in $L^\infty (\mathbb{R}^{d+1})$ and
$$
\average_{Q_\theta} \frac{\partial u_k}{\partial x_\ell} \to \average_{Q_\theta}
\frac{\partial u}{\partial x_\ell},
$$
we obtain
\begin{equation}\label{5.1-5}
\sup_{(x,t)\in Q_\theta}
\big| u(x,t)-u(0,0)-x_\ell \average_{Q_\theta} \frac{\partial u}{\partial x_\ell} \big|
\ge \theta^{3/2}.
\end{equation}
Also, since $u_k \to u$ weakly in $L^2(Q_1)$ and $\| u_k\|_{L^\infty(Q_1)} \le 1$, we have
$\int_{Q_1} |u|^2 \le |Q_1|$. 
It then follows from (\ref{5.1-5}) and (\ref{5.1-2}) that
$$
\theta^{3/2}
\le C_0 \theta^2 \left\{ \average_{Q_{1/2}} |u|^2\right\}^{1/2}
\le 2^{(d+2)/2} C_0 \theta^2 \left\{ \average_{Q_1} |u|^2\right\}^{1/2}
\le 2^{d} C_0 \theta^2,
$$
which is in contradiction with the choice of $\theta$.
This completes the proof.
\end{proof}

\begin{lemma}\label{lemma-5.2}
Let $\varep_0$ and $\theta$ be given by Lemma \ref{lemma-5.1}.
Suppose that $\big(\partial_t +\mathcal{L}_\varep\big) u_\varep =0$ in $Q_1$
and $0<\varep<\theta^{k-1}\varep_0$ for some $k\ge 1$.
Then there exist constants $b(\varep, k)\in \mathbb{R}$ and $E(\varep, k)
=(E_\ell^\beta (\varep, k)) \in \mathbb{R}^{dm}$, such that
$$
\left\{
\aligned
& |b(\varep, k)|\le C |E(\varep, k-1)|,\\
& |E(\varep, k)|\le C \left\{ 1+\theta^{1/2} +\cdots + \theta^{(k-1)/2}\right\} \| u_\varep \|_{L^\infty (Q_1)},
\endaligned
\right.
$$
and
\begin{equation}\label{5.2-1}
\aligned
\sup_{(x,t)\in Q_{\theta^k}}
\big| u_\varep (x,t) - u_\varep (0,0)-\varep b(\varep, k)
- & \left\{
P_\ell^\beta (x) + \varep \chi_\ell^\beta (x/\varep, t/\varep^2) \right\}  E_\ell^\beta (\varep, k) \big|\\
&\le
\theta^{3k/2} \| u_\varep\|_{L^\infty(Q_1)},
\endaligned
\end{equation}
where $C$ depends only on $d$, $m$, $\mu$, $\lambda$, and $\tau$.
\end{lemma}

\begin{proof}
As in the case of Lemma \ref{lemma-4.2}, Lemma \ref{lemma-5.2}
is proved by an induction argument on $k$. First, set $E(\varep, 0)=0$.
If $k=1$, we may use Lemma \ref{lemma-5.1} to see that the estimate (\ref{5.2-1})
holds with 
$$
b(\varep, 1) =0 \quad \text{ and } \quad
E_\ell^\beta (\varep, 1) =\average_{Q_\theta} \frac{\partial u_\varep^\beta}{\partial x_\ell}.
$$

Suppose now that the estimate (\ref{5.2-1}) holds for some $k\ge 1$.
Let $0<\varep<\theta^k \varep_0$ and $\big(\partial_t  +\mathcal{L}_\varep\big) u_\varep =0$ in $Q_1$.
Define
$$
w(x,t)=u_\varep (\theta^kx,\theta^{2k} t) -u_\varep (0,0)-\varep b(\varep, k)
-\theta^k \left\{ P_\ell^\beta (x) +\varep \theta^{-k} \chi_\ell^\beta (\theta^k x/\varep, \theta^{2k} t/\varep^2) \right\}
 E_\ell^\beta (\varep, k).
$$
Then, by the rescaling property (\ref{rescaling}),
$$
\left( \partial_t +\mathcal{L}_{\frac{\varep}{\theta^k}} \right) w =0 \quad \text{ in } Q_1.
$$
Since $\varep/\theta^k<\varep_0$, it follows from Lemma \ref{lemma-5.1} that
\begin{equation}\label{5.2-2}
\aligned
\sup_{(x,t)\in Q_\theta}
\big| w(x,t) -w(0,0) -& \left\{ P_\ell^\beta (x) +\varep\theta^{-k} \chi_\ell^\beta (\theta^kx/\varep, \theta^{2k} t/\varep^2)\right\}
 \average_{Q_\theta} \frac{\partial w^\beta}{\partial x_\ell} \big|\\
&\le \theta^{3/2} \| w\|_{L^\infty(Q_1)}.
\endaligned
\end{equation}
Note that by the induction assumption,
$$
\aligned
&\| w\|_{L^\infty(Q_1)}\\
&=\sup_{(x,t)\in Q_{\theta^k}}
\big| u_\varep (x,t) -u_\varep (0,0)-\varep b(\varep, k)
-  \left\{ P_\ell^\beta (x) +\varep \chi_\ell^\beta (x/\varep, t/\varep^2) \right\} E_\ell^\beta (\varep, k) \big|\\
&\le \theta^{3k/2} \| u_\varep \|_{L^\infty (Q_1)}.
\endaligned
$$
In view of (\ref{5.2-2}) this yields
$$
\aligned
\sup_{Q_{\theta^{k+1}}}
\big| u_\varep (x,t)-u_\varep(0,0) -\varep b(\varep, k+1)
-& \left\{ P_\ell^\beta (x) + \varep \chi_\ell^\beta (x/\varep, t/\varep^2) \right\} E_\ell^\beta (\varep, k+1) \big|\\
&\le \theta^{3(k+1)/2} 
\| u_\varep\|_{L^\infty (Q_1)},
\endaligned
$$
where we have set
$$
\aligned
b(\varep, k+1)  & =-\chi_\ell^\beta (0,0) E_\ell^\beta(\varep, k),\\
E_\ell^\beta (\varep, k+1) & =E_\ell^\beta (\varep, k) + \theta^{-k}
\average_{Q_\theta} \frac{\partial w^\beta }{\partial x_\ell}.
\endaligned
$$
Since $\|\chi\|_\infty\le C$, we see that $|b(\varep, k+1)|\le C | E(\varep, k)|$.

Finally, we observe that by the divergence theorem,
$$
\big| \average_{Q_\theta} \frac{\partial w}{\partial x_\ell }\big|
\le C \| w\|_{L^\infty(Q_\theta)}
\le C \theta^{3k/2} \| u_\varep \|_{L^\infty(Q_1)}.
$$
Hence,
$$
\aligned
|E(\varep, k+1)| & \le |E(\varep, k)| + C \theta^{k/2} \| u_\varep \|_{L^\infty(Q_1)}\\
&\le C \left\{ 1+\theta^{1/2} +\cdots + \theta^{k/2} \right\} \| u_\varep \|_{L^\infty(Q_1)}.
\endaligned
$$
This completes the induction argument and thus the proof.
\end{proof}

\begin{remark}\label{remark-5.1}
{\rm An inspection of the proof of Lemmas \ref{lemma-5.1} and \ref{lemma-5.2} shows that both lemmas
continue to hold under the weaker smoothness condition $A\in VMO_x$.
Indeed, the proof uses the interior H\"older estimates and rely on the fact that the correctors $\chi_\ell^\beta$ are bounded.
Furthermore, if $A^\# (r)\le \omega(r)$, where $\omega(r)$ is an increasing function
on $(0,1)$ such that $\lim_{r\to 0^+} \omega (r)=0$, 
then the constants $\theta$, $\varep_0$ and $C$ in Lemmas \ref{lemma-5.1} and \ref{lemma-5.2}
depend at most on $d$, $m$, $\mu$, and the function $\omega(r)$.
This observation will be used in the next section.
}
\end{remark}

\begin{proof}[\bf Proof of Theorem \ref{theorem-5.1}]
We will show that if $\big(\partial_t + \mathcal{L}_\varep \big) u_\varep =0 $ in $Q=Q_r (x_0, t_0)$, then
\begin{equation}\label{5.3-1}
|\nabla u_\varep (x_0, t_0)| \le C r^{-1} \| u_\varep \|_{L^\infty(Q)}.
\end{equation}
This, together with the $L^\infty$ estimate (\ref{L-infty}) in Theorem \ref{theorem-4.1}, implies that
if $\big(\partial_t +\mathcal{L}_\varep\big) u_\varep =0$ in $2Q$, then
$$
\|\nabla u_\varep \|_{L^\infty(Q)} \le \frac{C}{r}
\left\{ \average_{2Q} |u_\varep|^2\right\}^{1/2},
$$
where $Q=Q_r (x_0, t_0)$. 
It then follows from the inequality (\ref{Poincare}) that 
$$
\average_{Q_\rho (x,t)}
|u_\varep - \average_{Q_\rho(x,t)} u_\varep|^2
  \le C \rho^2 \average_{Q_\rho (x,t)} |\nabla u_\varep|^2\\
  \le C \left(\frac{\rho}{r}\right)^2 \average_{2Q} |u_\varep|^2,
 $$
 for any $(x,t)\in Q$ and $0<\rho<r/2$.
 By Lemma \ref{Cam-Lemma} we obtain
 $$
 \| u_\varep\|_{C^{1,1/2}(Q)} \le \frac{C}{r} \left\{ \average_{2Q} |u_\varep|^2\right\}^{1/2},
 $$
 for any weak solution of $\big(\partial_t +\mathcal{L}_\varep \big) u_\varep =0 $ in $2Q$.

To prove (\ref{5.3-1}), by translation and rescaling, we may assume that $(x_0, t_0)=(0,0)$ and $r=1$.
Let $\varep_0$ and $\theta$ be the constants given by Lemma \ref{lemma-5.1}.
We may  assume that $0<\varep<\theta\varep_0$, as the case $\varep\ge \theta\varep_0$ follows directly from the
standard Lipschitz estimates for parabolic systems in divergence form with H\"older continuous coefficients.

Suppose now that $\big(\partial_t +\mathcal{L}_\varep \big) u_\varep =0$ in $Q_1$ and $0<\varep< \theta\varep_0$.
We need to show that
\begin{equation}\label{5.3-4}
|\nabla u_\varep (0, 0)|\le C \| u_\varep \|_{L^\infty(Q_1)}.
\end{equation}
This will be done by using Lemma \ref{lemma-5.2} and a blow-up argument.

Choose $k\ge 1$ so that $\theta^{k+1} \varep_0\le \varep< \theta^k \varep_0$.
It follows from Lemma \ref{lemma-5.2} that
\begin{equation}\label{5.3-5}
\sup_{(x,t)\in Q_{\theta\varep/\varep_0}} |u_\varep (x,t) -u_\varep (0,0)|
\le \sup_{(x,t)\in Q_{\theta^{k+1}} } |u_\varep (x,t) -u_\varep (0,0)|
\le C \varep \|u_\varep\|_{L^\infty(Q_1)},
\end{equation}
where we have used the fact
$$
|E(\varep, k+1)| +|b(\varep, k+1)|
\le C \| u_\varep \|_{L^\infty(Q_1)}.
$$
Let 
$$
w(x,t)=\frac{ u_\varep (\varep x, \varep^2 t) -u_\varep (0,0)}{\varep}.
$$
Note that $\big(\partial_t +\mathcal{L}_1 \big) w =0 $ in $Q_{2/\varep_0}$.
By the standard local regularity theory for $\partial_t  +\mathcal{L}_1$, we obtain
\begin{equation}\label{5.3-7}
|\nabla w (0, 0)|\le C \| w\|_{L^\infty (Q_{\theta/\varep_0})}.
\end{equation}
It follows from (\ref{5.3-7}) and (\ref{5.3-5}) that
$$
|\nabla u_\varep (0, 0)|
=|\nabla w (0, 0)|
\le C \varep^{-1} \sup_{(x,t)\in Q_{\theta\varep/\varep_0}}
|u_\varep (x,t)-u_\varep (0,0)|
\le
C \| u_\varep\|_{L^\infty(Q_1)}.
$$
This completes the proof of (\ref{5.3-1}).
\end{proof}


\section{A real variable method and proof of Theorem \ref{main-theorem-1}}

\setcounter{equation}{0}

In this section we give the proof of Theorem \ref{main-theorem-1}.
We first treat the case where $u_\varep$ is a weak solution of 
$\big(\partial_t +\mathcal{L}_\varep \big) u_\varep =0$, i.e. $f=0$.
The general case will be handled by a real variable argument.

\begin{theorem}\label{theorem-10.1}
Suppose $A=A(x,t)$ satisfies the same conditions as in Theorem \ref{main-theorem-1}.
Let $u_\varep$ be a weak solution of $\big(\partial_t +\mathcal{L}_\varep \big) u_\varep
=0$ in $2Q$, where $Q=Q_r (x_0, t_0)$.
Then, for any $2<p<\infty$,
\begin{equation}\label{10-1}
\left(\average_Q |\nabla u_\varep|^p \right)^{1/p}
\le \frac{C_p}{r} 
\left(\average_{2Q} | u_\varep|^2 \right)^{1/2},
\end{equation}
where $C_p$ depends at most on $d$, $m$, $p$, $\mu$, and $A$.
\end{theorem}

\begin{proof}
By translation and dilation we may assume that $r=1$ and $(x_0, t_0)=(0,0)$.
We may also assume that $0<\varep<\theta \varep_0$, where $\varep_0$ and $\theta$ are constants given by
Lemma \ref{lemma-5.1} (see Remark \ref{remark-5.1}).
 This is because the case $\varep\ge \theta \varep_0$ follows from the $W^{1,p}$
estimates in \cite{Byun-2007, Krylov-2007} (see Remark \ref{remark-2.1}).

Let $w(x,t)=\varep^{-1}  u_\varep (\varep x, \varep^2 t)$ and $Q_\rho
 =Q_\rho (0,0)$. Since $\big( \partial_t +\mathcal{L}_1 \big) w=0$
in $Q_{2/\varep_0}$,
it follows from the local $W^{1,p}$ estimates for the operator $\partial_t +\mathcal{L}_1$ in \cite{Byun-2007, Krylov-2007} that
$$
\left(\average_{Q_{\theta\varep_0^{-1}}} |\nabla w|^p\right)^{1/p}
\le C \left(\average_{Q_{2\theta\varep_0^{-1}}} |w- w(0,0)|^2\right)^{1/2}
\le C \sup_{(x,t)\in Q_{2\theta\varep_0^{-1}} }|w(x,t)-w(0,0)|.
$$
It follows that
$$
\aligned
\int_{Q_{\varep \theta \varep_0^{-1}}} |\nabla u_\varep|^p
& \le C\, \varep^{-p} \,  |Q_{\varep\theta\varep_0^{{-1}}} |\sup_{(x,t)\in Q_{2\varep\theta\varep_0^{-1}} } {|u_\varep(x,t)  -u_\varep (0,0)|^p}\\
&\le C \varep^{d+2} \| u_\varep\|^p_{L^\infty(Q_1)},
\endaligned
$$
where we have used Lemma \ref{lemma-5.2} for the last inequality, as in (\ref{5.3-5}).
Clearly, by translation, this implies that 
$$
\int_{\widetilde{Q}} |\nabla u_\varep|^p \le C |\widetilde{Q}| \| u_\varep\|^p_{L^2(Q_2)},
$$
for any cylinder $\widetilde{Q}=Q_\rho (x_1, t_1)$
 of size $\rho=\varep\theta\varep_0^{-1}$ such that $\widetilde{Q}\subset (3/2)Q_2$.
By covering $Q_1$ with such cylinders, one may deduce that
$$
\int_{Q_1} |\nabla u_\varep|^p \le C \| u_\varep\|_{L^2(Q_2)}^p.
$$
This completes the proof.
\end{proof}

To handle the general case, where $u_\varep$ is a weak solution of
$\big(\partial_t +\mathcal{L}_\varep \big) u_\varep =\text{div} (f)$,
we use the following theorem.

\begin{theorem}\label{real-variable-theorem}
Let $\widetilde{Q}=Q_R (x_0,t_0)$ and $F\in L^2(4\widetilde{Q})$.
Let $q>2$ and $f\in L^p(4\widetilde{Q})$ for some $2<p<q$.
Suppose that for each $Q=Q_r (x,t) \subset 2\widetilde{Q}$ with $|Q|\le c_1 |\widetilde{Q}|$,
there exist two measurable functions $F_Q$ and $R_Q$ on $2Q$, such that
$|F|\le |F_Q| +|R_Q|$ on $2Q$,
\begin{equation}\label{real-1}
\aligned
\left(\average_{2Q} |R_Q|^q \right)^{1/q}
& \le C_1 \left\{ \left( \average_{c_2 Q} |F|^2 \right)^{1/2}
+ \sup_{4\widetilde{Q}\supset Q^\prime\supset Q} \left(\average_{Q^\prime} |f|^2\right)^{1/2} \right\},\\
\left(\average_{2Q} |F_Q|^2 \right)^{1/2}
& \le C_2 \sup_{4\widetilde{Q} \supset Q^\prime\supset Q} \left(\average_{Q^\prime} |f|^2\right)^{1/2},
\endaligned
\end{equation}
where $C_1$, $C_2>0$, $0<c_1<1$, and $c_2>2$. Then
$F\in L^p(\widetilde{Q})$ and
\begin{equation}\label{real-2}
\left(\average_{\widetilde{Q}} |F|^p\right)^{1/p}
\le C \left\{ \left(\average_{4\widetilde{Q}} |F|^2\right)^{1/2}
+\left(\average_{4\widetilde{Q}} |f|^p \right)^{1/p} \right\},
\end{equation}
where $C$ depends only on $d$, $C_1$, $C_2$, $c_1$, $c_2$, $p$, and $q$.
\end{theorem}

Theorem \ref{real-variable-theorem}, whose proof 
we omit, is the parabolic version of Theorem 3.2 in \cite{Shen-2007-boundary}, which was proved by using
a real-variable argument originated in \cite{Caffarelli-Peral} and further developed in \cite{Shen-2005-bounds}.
The argument in \cite{Shen-2007-boundary}, which is based on a Calder\'on-Zygmund decomposition and uses
the $L^p$ boundedness of the Hardy-Littlewood maximal functions, 
extends easily to the parabolic setting.

\begin{proof}[\bf Proof of Theorem \ref{main-theorem-2}]
Suppose that $\big( \partial_t +\mathcal{L}_\varep\big) u_\varep =\text{div} (f)$ in $4\widetilde{Q}$.
We will show that
\begin{equation}\label{10.2-1}
\left\{\average_{\widetilde{Q}} |\nabla u_\varep|^p \right\}^{1/p}
\le C \left\{\left(\average_{4\widetilde{Q}} |\nabla u_\varep|^2\right)^{1/2}
+\left(\average_{4\widetilde{Q}} |f|^p \right)^{1/p} \right\},
\end{equation}
By the energy estimates and a simple covering argument, it is easy to see that the estimate (\ref{10.2-1})
is equivalent to (\ref{1.1-1}).
We also point out that as in the case of local estimates (see (\ref{krylov-1})),
 the estimate (\ref{1.1-2}) follows from (\ref{1.1-1})
by Lemmas \ref{Poincare-lemma} and \ref{Cam-Lemma}.

To prove (\ref{10.2-1}), we let $q=p+1$ and apply Theorem \ref{real-variable-theorem} to
$F=|\nabla u_\varep|$.
For each $Q=Q_r (x,t)\subset 2\widetilde{Q}$ with $|Q|\le (100)^{-d-2} |\widetilde{Q}|$, we set
$$
F_Q =\nabla v_\varep \quad \text{ and } \quad R_Q =\nabla (u_\varep -v_\varep),
$$
where $v_\varep$ solves the initial-Dirichlet problem $\big(\partial_t +\mathcal{L}_\varep\big) v_\varep =\text{div} (f)$
in $4Q$ with zero initial and boundary data. Clearly, $|F|\le |F_Q| +|R_Q|$ on $4Q$.
By the well-known energy estimates,
\begin{equation}\label{10.2-3}
\average_{4Q} |\nabla v_\varep|^2 \le C \average_{4Q} |f|^2,
\end{equation}
where $C$ depends only on $d$, $m$, and $\mu$.
This gives the second inequality in (\ref{real-1}).

To verify the first inequality in (\ref{real-1}), we note that
$$
\big(\partial_t +\mathcal{L}_\varep \big) (u_\varep -v_\varep) =0 \quad \text{ in } 4Q.
$$
It follows from Theorem \ref{theorem-10.1}  that
$$
\aligned
\left(\average_{2Q} |R_Q|^q \right)^{1/q}
&\le C \left(\average_{4Q} |R_Q|^2 \right)^{1/2}\\
&\le C \left(\average_{4Q} |F|^2\right)^{1/2}
+C \left(\average_{4Q} |\nabla v_\varep|^2 \right)^{1/2}\\
&\le 
C \left(\average_{4Q} |F|^2\right)^{1/2}
+C \left(\average_{4Q} |f|^2 \right)^{1/2},
\endaligned
$$
where we have used (\ref{10.2-3}) for the last inequality.
This gives the second inequality in (\ref{real-2}). As a result,  the desired estimate (\ref{10.2-1})
follows by Theorem \ref{real-variable-theorem}.
\end{proof}


\section{Fundamental solutions and proof of Theorem \ref{main-theorem-2}}
\setcounter{equation}{0}

Suppose that $A$ satisfies  conditions (\ref{e1.3})-(\ref{e1.4}) and $A\in VMO_x$.
Let $u_\varep $ be a weak solution of $\big(\partial_t +\mathcal{L}_\varep \big) u_\varep =0$
in $Q_r=Q_r (x_0, t_0)$. It follows from Theorem \ref{theorem-4.1} that for any $0<\rho<r$,
\begin{equation}\label{6-0}
\average_{Q_\rho} |u_\varep -\average_{Q_\rho} u_\varep|^2
\le C \left(\frac{\rho}{r}\right)^{2\alpha} 
\average_{Q_r} |u_\varep -\average_{Q_r} u_\varep|^2
\end{equation}
where $0<\alpha<1$ and $C$ depends only on $d$, $m$, $\alpha$, $\mu$, and $A^\#$.
In view of Lemmas \ref{Cacci-Lemma} and \ref{Poincare-lemma} this implies that
\begin{equation}\label{6-1}
\average_{Q_\rho} |\nabla u_\varep|^2
\le C \left(\frac{\rho}{r} \right)^{2\alpha-2} 
\average_{Q_r} |\nabla u_\varep|^2, \quad \text{ for any } 0<\rho<r.
\end{equation}
Let $v_\varep$ be a weak solution of $\big( -\partial_t +\mathcal{L}_\varep^*\big) v_\varep=0$
in $Q_r^+$, where 
$$
Q_r^+=Q_r^+ (x_0, t_0)=B(x_0, r)\times (t_0, t_0+r^2)
$$
 and
$\mathcal{L}_\varep^* =-\text{div} \big( A^* (x/\varep, t/\varep^2) \nabla \big)$.
Then
$u_\varep (x,t)= v_\varep (x, 2t_0-t)$ is a solution of 
$\partial_t u_\varep -\text{div} \big( C(x/\varep, t/\varep^2)\nabla u_\varep\big) =0$ in $Q_r (x_0, t_0)$,
where $C(y,s)=A^* (y, s+ 2t_0/\varep^2)$.
Since $C(y,s)$ satisfies the same ellipticity, periodicity, and smoothness conditions as $A(y,s)$,
the estimate (\ref{6-1}) holds for this $u_\varep$. As a result, by a change of variables, we obtain
\begin{equation}\label{6-2}
\average_{Q_\rho^+} |\nabla v_\varep|^2
\le C \left(\frac{\rho}{r} \right)^{2\alpha-2} 
\average_{Q_r^+} |\nabla v_\varep|^2 \quad \text{ for any } 0<\rho<r.
\end{equation}
It is known that under the H\"older conditions (\ref{6-1})-(\ref{6-2}), the matrix of fundamental solutions
$\Gamma_\varep (x,t; y;s) =\big(\Gamma_{\varep, ij}^{\alpha\beta} (x,t; y,s) \big)$
for  the parabolic operator $\partial_t +\mathcal{L}_\varep$ in $\mathbb{R}^{d+1}$
 exists and satisfies the size estimate
 \begin{equation}\label{size-estimate}
 |\Gamma _\varep (x,t;y,s)|
 \le \frac{C}{(t-s)^{d/2}}  e^{-\frac{\kappa |x-y|^2}{t-s}},
 \end{equation}
 for any $t>s$ and $x,y\in \mathbb{R}^d$, where $\kappa$ and $C$ are positive constants depending
 only on $d$, $m$, $\mu$, and $A^\#$ (see \cite{ Cho-Dong-Kim-2008}).
 In particular, we have
 \begin{equation}\label{size-estimate-1}
 |\Gamma_\varep (x,t; y,s)|\le \frac{C}{\big( |x-y|+|t-s|^{1/2}\big)^d}.
 \end{equation}
 
Suppose now that $A\in \Lambda (\mu, \lambda, \tau)$.
It follows from Theorem \ref{theorem-5.1} and (\ref{size-estimate-1}) that
\begin{equation}\label{gradient-estimate}
|\nabla_x \Gamma_\varep(x,t; y,s)|
+|\nabla_y \Gamma_\varep (x,t; y,s)|
\le \frac{C}{ \big( |x-y| +|t-s|^{1/2} \big)^{d+1}},
\end{equation}
for any $(x,t), (y,s)\in \mathbb{R}^{d+1}$ and $t>s$. 
Note that as a function of $(x,t)$,
$\nabla_y \Gamma_\varep (x,t; y,s)$ is a solution 
of $\big( \partial_t +\mathcal{L}_\varep\big) u_\varep =0$ in $\mathbb{R}^{d+1} \setminus \{ (y,s)\}$.
Thus we may use the Lipschitz estimates in Theorem \ref{theorem-5.1} and (\ref{gradient-estimate})
to obtain
\begin{equation}\label{gradient-estimate-1}
|\nabla_x\nabla_y \Gamma_\varep (x,t;y,s)|
\le 
 \frac{C}{ \big( |x-y| +|t-s|^{1/2} \big)^{d+2}},
\end{equation}
for any $(x,t), (y,s)\in \mathbb{R}^{d+1}$ and $t>s$. 
This allows us to complete the proof of Theorem \ref{main-theorem-2}.

\begin{proof}[\bf Proof of Theorem \ref{main-theorem-2}]

Let $A\in \Lambda (\mu, \lambda, \tau)$.
Suppose that $\big(\partial_t +\mathcal{L}_\varep\big) u_\varep =F$ in $2Q$, where
$Q=Q_r (x_0, t_0)$ and $F\in L^p(2Q)$ for some $p>d+2$.
By dilation and translation we may assume that $r=1$ and $(x_0, t_0)=(0,0)$.
We need to show that
\begin{equation}\label{6.1-1}
\| u_\varep\|_{C^{1, 1/2} (Q)}
\le C_p \big\{ \| u_\varep\|_{L^2(2Q)} +\| F\|_{L^p(2Q)}\big\}.
\end{equation}
 
 To this end we choose  a scalar function $\psi \in C^\infty(\mathbb{R}^{d+1})$
 such that $\psi(x,t)=1$ on $Q$, and $\psi(x,t)=0$ if $|x|>3/2$ or $t<-(3/2)^2$.
 Using the representation of $u_\varep \psi$ by the matrix of fundamental solutions $\Gamma_\varep(x,t; y,s)$,
 one may deduce that if $(x,t)\in Q$, then
 $$
 \aligned
 |\nabla u_\varep(x,t)|
  \le  &  \int_{2Q} |\nabla_x \Gamma_\varep (x,t;y,s)| | F(y,s)\psi(y,s)| \, dyds
 +\int_{2Q} |\nabla_x \Gamma_\varep (x,t;y,s)| |u_\varep \partial_s \psi|\, dyds\\
& \quad +C \int_{2Q} |\nabla_x\Gamma_\varep (x,t; y,s)| |\nabla u_\varep | |\nabla \psi|\, dyds \\
& \quad+C \int_{2Q} |\nabla_x \nabla_y \Gamma(x,t; y,s)| |u_\varep| |\nabla \psi|\, dyds.
 \endaligned
 $$
In view of (\ref{size-estimate-1}), (\ref{gradient-estimate}) and (\ref{gradient-estimate-1}), 
this gives
$$
|\nabla u_\varep (x,t)|
\le C \int_{2Q} \frac{|F(y,s)|\, dyds }{\big( |x-y|+|t-s|^{1/2}\big)^{d+1}}
+C \int_{2Q} |u_\varep|
+C \int_{(3/2)Q} |\nabla u_\varep|.
$$
By H\"older's inequality and the energy estimate (\ref{Caccio}), it follows that
\begin{equation}\label{6.1-3}
\|\nabla u_\varep\|_{L^\infty(Q)}
\le C \big\{ \| F\|_{L^p(2Q)} + \| u_\varep\|_{L^2(2Q)}\big\}.
\end{equation}
Finally, one may use Lemmas \ref{Poincare-lemma} and \ref{Cam-Lemma} to deduce 
the desired estimate (\ref{6.1-1}) from (\ref{6.1-3}).
\end{proof}


\section{Boundary H\"older estimates}
\setcounter{equation}{0}

Let $\Omega$ be a bounded $C^1$ domain in $\mathbb{R}^d$.
For $x_0\in \overline{\Omega}$, $t_0\in \mathbb{R}$, and $0<r<r_0=\text{diam}(\Omega)$, let
\begin{equation}\label{7.1}
\aligned
& \Omega_r(x_0, t_0)= \big[B(x_0,r)\cap \Omega\big]\times (t_0-r^2, t_0),\\
& \Delta_r (x_0, t_0) =\big[B(x_0, r)\cap \partial\Omega\big] \times (t_0-r^2, t_0).
\endaligned
\end{equation}
Throughout this section we shall assume that
$A=A(x,t)$ satisfies (\ref{e1.3})-(\ref{e1.4}) and $A\in VMO_x$.

The goal of this section is to establish the following theorem.

\begin{theorem}\label{theorem-7.1}
Let $0<\alpha<1$. Suppose that $u_\varep$ is a weak solution of
\begin{equation}\label{7.1-1}
\big(\partial_t +\mathcal{L}_\varep\big) u_\varep =0
\quad \text{ in } \Omega_{2r} (x_0, t_0) \quad \text{ and } \quad
u_\varep =0 \quad \text{ on } \Delta_{2r} (x_0, t_0),
\end{equation}
for some $x_0\in \partial\Omega$, $t_0\in \mathbb{R}$, and $0<r<\text{\rm diam} (\Omega)$.
Then
\begin{equation}\label{7.1-2}
\| u_\varep\|_{C^{\alpha, \alpha/2} (\Omega_r (x_0, t_0))}
\le C r^{-\alpha}
 \left(\average_{\Omega_{2r} (x_0, t_0)} |u_\varep|^2\right)^{1/2}
\end{equation}
where $C$ depends only on $d$, $m$, $\alpha$, $\mu$, $A$, and $\Omega$.
\end{theorem}

Let $\psi:\mathbb{R}^{d-1}\to \mathbb{R}$ be a $C^1$ function with compact support.
It will be convenient to assume that
\begin{equation}\label{condition-psi}
\aligned
\psi (0) =0\text{ and } & \|\nabla \psi\|_{L^\infty(\mathbb{R}^{d-1})}\le M_0,\\
\sup\big\{ |\nabla\psi (x^\prime)-\nabla \psi(y^\prime)|: & \ x^\prime, y^\prime\in \mathbb{R}^{d-1} \text{ and }
|x^\prime-y^\prime|<r \big\} \le \eta (r),
\endaligned
\end{equation}
where $\eta(r)$ is a fixed bounded increasing function on $(0, \infty)$ such that
$\lim_{r\to 0^+} \eta (r)=0$.
For $r>0$, define
\begin{equation}\label{7.3}
\aligned
& D_r=D_r (\psi)
=\big\{ (x^\prime, x_d, t)\in \mathbb{R}^{d+1}:\
|x^\prime|<r, \ \psi(x^\prime) < x_d < \psi(x^\prime) + 10 (M_0+1) r \\
& \quad \qquad\qquad\qquad\qquad\qquad\qquad\qquad  \text{ and } -r^2< t<0 \big\},\\
& I_r=I_r (\psi)
=\big\{ (x^\prime, \psi(x^\prime), t)\in \mathbb{R}^{d+1}: \ |x^\prime|<r \text{ and } -r^2< t<0\big\}.
\endaligned
\end{equation}

By a change of the coordinate system and using Campanato's characterization of H\"older spaces, one may
deduce Theorem \ref{theorem-7.1} from the following theorem.

\begin{theorem}\label{theorem-7.2}
Let $0<\alpha<1$.
Suppose that
\begin{equation}\label{7.2-1}
\big(\partial_t +\mathcal{L}_\varep\big) u_\varep =0
\quad \text{ in } D_{r}  \quad \text{ and } \quad
u_\varep =0 \quad \text{ on } I_{r}
\end{equation}
for some $r>0$. 
Then, for any $0<\rho<r$,
\begin{equation}\label{7.2-2}
\left(\average_{D_\rho} |u_\varep|^2\right)^{1/2}
\le C \left( \frac{\rho}{r}\right)^\alpha
\left(\average_{D_r}  |u_\varep|^2\right)^{1/2},
\end{equation}
where $C$ depends only on $d$, $m$, $\alpha$, $\mu$, $\omega (r)$ in (\ref{VMO-x}), and $(M_0, \eta( r))
$ in (\ref{condition-psi}).
\end{theorem}

As in the case of interior estimates, Theorem \ref{theorem-7.2} will be proved by a compactness argument.

\begin{lemma}\label{lemma-7.4}
Let $\{ A_k(y,s)\}$ be a sequence of matrices satisfying conditions (\ref{e1.3})-(\ref{e1.4}) and $\{\psi_k\}$
a sequence of $C^{1}$ functions satisfying (\ref{condition-psi}).
Suppose that
$$
\partial_t u_k -\text{\rm div} \big( A_k (x/\varep_k, t/\varep_k^2)\nabla u_k \big) =0 \quad 
\text{ in } D_r(\psi_k) \quad \text{ and } \quad
u_k=0 \quad \text{ on } I_r(\psi_k),
$$
where
$\varep_k \to 0$ and
\begin{equation}\label{7.4-1}
\|u_k\|_{L^2(D_r(\psi_k))}
+\| \nabla u_k\|_{L^2(D_r(\psi_k))} \le C.
\end{equation}
Then there exist subsequences of $\{ \psi_k\}$ and $\{ u_k\}$, which we still denote by the same notation,
and a function $\psi$ satisfying (\ref{condition-psi}),
$u\in L^2(D_r(\psi))$, and a constant matrix $A^0$ satisfying (\ref{ellipticity-1}), such that
\begin{equation}\label{7.4-2}
\left\{
\aligned
& \psi_k \to \psi \text{ in } C^1(|x^\prime|<r),\\
& u_k(x^\prime, x_d-\psi_k(x^\prime), t) \to u(x^\prime, x_d-\psi(x^\prime), t)
 \text{ strongly in } L^2(E_r \times (-r^2, 0)),
\endaligned
\right.
\end{equation}
where $E_r = \big\{ (x^\prime, x_d): |x^\prime|<r \text{ and } 0<x_d<10 (M_0+1)r \big\}$,
 and $u$ is a solution of
\begin{equation}\label{7.4-3}
\partial_t u -\text{\rm div}\big(A^0\nabla u\big) =0 \quad \text{ in } D_r(\psi) \quad\text{ and } \quad
u=0 \quad \text{ on } I_r(\psi).
\end{equation}
\end{lemma}

\begin{proof}
By passing to a subsequence we may clearly assume that $\widehat{A_k}\to A^0$.
By the Arzel\'a-Ascoli Theorem we may also assume that
$\psi_k \to \psi$ in $C^1(|x^\prime|<r)$ for some function $\psi$ satisfying (\ref{condition-psi}).
Let 
$$
v_k(x^\prime, x_d, t)=u_k(x^\prime, x_d-\psi_k (x^\prime), t).
$$
Note that  
$$
\aligned
& \{ v_k\} \text{ is bounded in } L^2(-r^2, 0; W^{1,2}(E_r)),\\
& \{ \partial_t v_k \} \text{  is bounded in } L^2(-r^2, 0; W^{-1, 2}(E_r)).
\endaligned
$$
 Hence, by Lemma \ref{embedding-lemma},
we may also assume that $v_k\to v$ strongly in $L^2(E_r\times (-r^2, 0))$
and $\nabla v_k \to \nabla v$ weakly in $L^2(E_r \times (-r^2,0))$.

Let 
$$
u(x^\prime, x_d, t) =v(x^\prime, x_d+\psi(x^\prime), t).
$$
Note that $u=0$ on $I_r (\psi)$.
It is not hard to check that
$u_k \to u$ strongly in $L^2(Q)$ and $\nabla u_k \to \nabla u$ weakly in $L^2(Q)$, for any
$Q\subset \subset D_r(\psi)$.
By Theorem \ref{homo-theorem-2} this implies that
$u$ is a solution of $\partial_t u -\text{div}\big(A^0 \nabla u\big)=0$ in $Q$
for any $Q=
\Omega \times (T_0, T_1)\subset\subset D_r (\psi)$.
Since $\nabla u\in L^2(D_r (\psi))$,
$u$ is also a  solution of $\partial_t u -\text{div}\big(A^0\nabla u\big) =0$ in 
$D_r(\psi)$.
\end{proof}

\begin{lemma}\label{lemma-7.5}
Let $0<\alpha<1$.
Then there exist constants $\varep_0>0$ and $\theta\in (0,1/4)$, depending only on $d$, $m$, $\mu$,
 and $(M_0, \eta (r))$ in (\ref{condition-psi}),
such that
\begin{equation}\label{7.5-1}
\average_{D_\theta (\psi)} |u_\varep|^2 \le \theta^{2\alpha},
\end{equation}
whenever $0<\varep<\varep_0$,
\begin{equation}\label{7.5-2}
\left\{
\aligned
\big( \partial_t +\mathcal{L}_\varep \big) u_\varep  &=0&\quad & \text{ in } D_1 (\psi),\\
u_\varep & = 0& \quad & \text{ on } I_1(\psi),
\endaligned
\right.
\end{equation}
and
\begin{equation}\label{7.5-3}
 \average_{D_1 (\psi)} |u_\varep|^2 \le 1.
\end{equation}
\end{lemma}

\begin{proof} Let $\sigma =(1+\alpha)/2\in (\alpha, 1)$.
The lemma is proved by contradiction, using Lemma \ref{lemma-7.4} and the following regularity estimate:
\begin{equation}\label{7.5-4}
\average_{D_r (\psi)} |w|^2 \le C_0 r^{2\sigma} \quad \text{ for any } 0<r<1/4,
\end{equation}
whenever
\begin{equation}\label{7.5-5}
\left\{
\aligned
\partial_t w -\text{div} \big(A^0\nabla w\big)  & =0 \ \ \ \  \text{ in } D_{1/2} (\psi),\\
 w & =0 \ \ \ \ \text{ on } I_{1/2}(\psi),\\
  \average_{D_{1/2}(\psi)} |w|^2 & \le 1,
 \endaligned
 \right.
 \end{equation}
and $A^0$  in (\ref{7.5-5}) is a constant matrix satisfying (\ref{ellipticity-1}). We remark that the estimate (\ref{7.5-4})
follows from the boundary H\"older estimate:
$$
\| w\|_{C^{\sigma, \sigma/2} (D_{1/4}(\psi))}
\le C \| w\|_{L^2(D_{1/2}(\psi))} $$
for second-order parabolic systems with constant coefficients in $C^1$ cylinders, and
 the constant $C_0$ in (\ref{7.5-4})
depends only on $d$, $m$, $\mu$, and $(M_0, \eta (r))$ in (\ref{condition-psi}).
 
Choose $\theta\in (0,1/4)$ so small that $2^d C_0 \theta^{2\sigma} < \theta^{2\alpha}$.
We claim that for this $\theta$, there exists some $\varep_0>0$, depending only on $d$, $m$, $\mu$, 
and $(M_0, \eta(r))$,
such that the estimate (\ref{7.5-1}) holds, whenever $0<\varep<\varep_0$ and $u_\varep$ satisfies
(\ref{7.5-2}) and (\ref{7.5-3}).

Suppose this is not the case. Then there exist sequences $\{ \varep_k\}$, $\{A_k\}$, $\{ \psi_k\}$, and $\{ u_k\}$,
such that $\varep_k \to 0$, $A_k$ satisfies (\ref{e1.3})-(\ref{e1.4}), $\psi_k$ satisfies (\ref{condition-psi}),
\begin{equation}\label{7.5-6}
\left\{
\aligned
\partial_t u_k -\text{div} \big( A_k (x/\varep_k, t/\varep_k^2)\nabla u_k\big) & =0 \ \ \ \ \text{ in } D_1(\psi_k),\\
u_k & =0 \ \ \ \ \text{ on } I_1(\psi_k),\\
\average_{D_1(\psi_k)} |u_k|^2 & \le 1,
\endaligned
\right.
\end{equation}
and
\begin{equation}\label{7.5-7}
\average_{D_\theta(\psi_k)} |u_k|^2 > \theta^{2\alpha}.
\end{equation}
Since $\| u_k\|_{L^2(D_1(\psi_k))}$  is bounded,
by  the energy estimates, it follows that $\| \nabla u_k\|_{L^2(D_{1/2}(\psi_k))}$ is bounded.
This allows us to use Lemma \ref{lemma-7.4}.

Indeed, in view of Lemma \ref{lemma-7.4}, by passing to a subsequence, we may assume that
$u_k \to u $ strongly in $L^2(Q)$ for any $Q\subset\subset D_{1/2}(\psi)$, and $u$ is a  solution of
$\partial_t u -\text{div} \big(A^0\nabla u\big) =0$ in $D_{1/2}(\psi)$ and
$u=0$ on $I_{1/2}(\psi)$, where $A^0$ is a constant matrix satisfying (\ref{ellipticity-1}) and
$\psi$ satisfies (\ref{condition-psi}).
It is not hard to see that
$$
\average_{D_{1/2} (\psi)} |u|^2 \le 1.
$$

Finally, it follows from (\ref{7.5-4})-(\ref{7.5-5}) that
$$
\average_{D_\theta (\psi)} |u|^2 \le C_0 \theta^{2\sigma}.
$$
However, since $ u_k (x^\prime, x_d-\psi_k (x^\prime), t )\to u(x^\prime, x_d -\psi(x^\prime), t)$
strongly in $L^2(E_{1/2}\times (-1/4,0))$, we may deduce from (\ref{7.5-7}) that
$$
\average_{D_\theta (\psi)} |u|^2\ge \theta^{2\alpha}.
$$
This leads to $\theta^{2\alpha}\le C_0 \theta^{2\sigma}$, which is in contradiction with the choice of $\theta$.
The proof is complete.
\end{proof}

\begin{lemma}\label{lemma-7.6}
Fix $0<\alpha<1$.
Let $\varep_0$ and $\theta$ be the constants given by Lemma \ref{lemma-7.5}.
Suppose that $\big(\partial_t +\mathcal{L}_\varep\big) u_\varep =0$
in $D_1(\psi)$ and $u_\varep =0$ on $I_1(\psi)$.
Then, if $0<\varep<\theta^{k-1}\varep_0$ for some $k\ge 1$,
\begin{equation}\label{7.6-1}
\average_{D_{\theta^k} (\psi)} |u_\varep|^2
\le \theta^{2k\alpha} 
\average_{D_1(\psi)} |u_\varep|^2\end{equation}
\end{lemma}

\begin{proof}
The lemma is proved by induction.
We first note that the case $k=1$ is given by Lemma \ref{lemma-7.5}.
Suppose now that the lemma holds for some $k\ge 1$.
Let $0<\varep<\theta^k \varep_0$.
We apply Lemma \ref{lemma-7.5} to the function $w(x,t) =u_\varep (\theta^k x, \theta^{2k}t)$ in 
$D_1(\psi_k)$, where $\psi_k (x^\prime) =\theta^{-k} \psi(\theta^kx^\prime)$.
Observe that 
$$
\big( \partial_t +\mathcal{L}_{\frac{\varep}{\theta^k}} \big) w= 0 \quad \text{ in } D_1 (\psi_k),
$$
$\theta^{-k}\varep<\varep_0$, and $\psi_k$ satisfies the condition (\ref{condition-psi}).
It follow from Lemma \ref{lemma-7.5} that
$$
\aligned
\average_{D_{\theta^{k+1}} (\psi)} |u_\varep|^2
& =\average_{D_\theta (\psi_k)} | w|^2
\le \theta^{2\alpha}
\average_{D_1(\psi_k)} |w|^2
 =\theta^{2\alpha} 
 \average_{D_{\theta^k} (\psi)} |u_\varep|^2\\
& \le \theta^{2 (k+1)\alpha}
 \average_{D_1(\psi)} |u_\varep|^2,
\endaligned
$$
where we have used the induction assumption in the last inequality. 
This complete the proof.
\end{proof}

\begin{proof}[\bf Proof of Theorem \ref{theorem-7.2}]

By rescaling we may assume that $r=1$.
We may also assume that $\varep<\varep_0$, since the case $\varep\ge \varep_0$ follows from the boundary
H\"older estimates for second-order parabolic systems in divergence form with $VMO_x$ coefficients
in $C^1$ cylinders. Such estimates may be deduced from the $W^{1,p}$
estimates in \cite{Byun-2007}.
We may further assume that 
$ \| u_\varep\|_{L^2(D_1(\psi))} \le 1.
$
Under these assumptions we will show that
\begin{equation}\label{7.7-1}
\average_{D_\rho (\psi)} |u_\varep|^2 \le C \rho^{2\alpha},
\end{equation}
for any $0<\rho<1$.

To prove (\ref{7.7-1}), we first consider the case $\rho\ge \varep/\varep_0$.
Choose $k\ge 1$ such that $\theta^k\le \rho<\theta^{k-1}$.
Since $\varep\le \varep_0\rho<\varep_0 \theta^{k-1}$,
it follows from Lemma \ref{lemma-7.6} that
$$
\average_{D_\rho(\psi)} |u_\varep|^2
\le C \average_{D_{\theta^{k-1}} (\psi)} |u_\varep|^2
\le C \theta^{2k\alpha} 
\le C \rho^{2\alpha}.
$$

Next suppose that $0<\rho< \varep/\varep_0$.
Let $w(x,t)=u_\varep (\varep x, \varep^2 t)$.
Then $\big(\partial_t +\mathcal{L}_1 \big) w=0$ in $D_{1/\varep_0} (\psi_\varep)$, 
where $\psi_\varep (x^\prime)=\varep^{-1}\psi(\varep x^\prime)$.
By the boundary H\"older estimates for the parabolic operator $\partial_t +\mathcal{L}_1$ in $C^1$
cylinders,
we see that
$$
\average_{D_{\rho /\varep} (\psi_\varep)} |w|^2 
\le C \left(\frac{\rho}{\varep}\right)^{2\alpha}
 \average_{D_{1/\varep_0} (\psi_\varep)} |w|^2,
$$
where $C$ depends at most on $d$, $m$, $\alpha$, $\mu$, $A^\#$,
and $(M_0, \eta(r))$ in (\ref{condition-psi}).
This yields
$$
\average_{D_\rho(\psi)} |u_\varep|^2 
\le C \left(\frac{ \rho}{\varep}\right)^{2\alpha}
 \average_{D_{\varep/\varep_0} (\psi)} 
|u_\varep|^2 
\le C \rho^{2\alpha},
$$
where we have used the estimate (\ref{7.7-1}) for the case $\rho=\varep/\varep_0$ 
in the last inequality.
This completes the proof of Theorem \ref{theorem-7.2}.
\end{proof}

Before we give the proof of Theorem \ref{theorem-7.1} we make a few remarks
on the  change of the coordinate system by a rotation and its effects on the operator $\mathcal{L}_\varep$.
It is clear that a general rotation would destroy the $Y$-periodicity of the coefficients of the operator.
However, the following theorem, whose proof may be found in \cite{Schmutz-2008},  may be used to resolve this issue.

\begin{theorem}\label{rotation-theorem}
Let $\mathcal{O}=\big(\mathcal{O}_{ij}\big)$ be a $d\times d$ orthogonal matrix.
For any $\delta>0$, there exists a $d\times d$ orthogonal matrix $\mathcal{R}=\big(\mathcal{R}_{ij}\big)$
with rational entries such that
(1) $\| \mathcal{O}-\mathcal{R}\|_\infty =\sup_{ij} |\mathcal{O}_{ij} -\mathcal{R}_{ij} | <\delta$; 
(2) each entry of $\mathcal{R}$ has denominator less than a constant depending only on $d$ and $\delta$.
\end{theorem}

Consider
$$
 \partial_t u_\varep -\text{div} \big( A(x/\varep, t/\varep^2)\nabla u_\varep\big) =F(x,t)
 \quad \text{ in } \big[\Omega \cap B(0, r_0)\big] \times (-r_0^2 , 0),
 $$
where $\Omega$ is a bounded Lipschitz domain in $\mathbb{R}^d$, $0\in \partial\Omega$, and
$r_0>0$ is  small.
There exists an orthogonal matrix $\mathcal{O}$ such that
$$
\Omega_1 \cap B(0, r_0) =\big\{ (y^\prime, y_d)\in \mathbb{R}^d: \ y_d >\psi_1 (y^\prime) \big\} \cap B(0, r_0),
$$
where 
$
\Omega_1 =\big\{ y\in \mathbb{R}^d: \ y= \mathcal{O} x \text{ for some } x\in \Omega\big\},
$
and $\psi_1$ is a Lipschitz function in $\mathbb{R}^{d-1}$ such that $\psi_1 (0)=0$ and
$\|\nabla \psi_1\|_\infty\le M$.
Observe that if $\mathcal{R}$ is an orthogonal matrix such that $\|\mathcal{R}-\mathcal{O}\|_\infty
<\delta$, where $\delta>0$, depending only on $d$ and $M$, is sufficiently small, then
$$
\Omega_2 \cap B(0, r_0/2) =\big\{ (y^\prime, y_d)\in \mathbb{R}^d: \ y_d >\psi_2 (y^\prime) \big\} \cap B(0, r_0/2),
$$
where 
$
\Omega_2 =\big\{ y\in \mathbb{R}^d: \ y= \mathcal{R} x \text{ for some } x\in \Omega\big\},
$
and $\psi_2$ is a Lipschitz function in $\mathbb{R}^{d-1}$ such that $\psi_2 (0)=0$ and
$\|\nabla \psi_2\|_\infty\le M+1$.
Since $\|\mathcal{O}-\mathcal{R}\|_\infty \le C_d \| \mathcal{O}^{-1} -\mathcal{R}^{-1}\|_\infty$,
in view of Theorem \ref{rotation-theorem}, we may choose $\mathcal{R}=\big(\mathcal{R}_{ij}\big)$ in such a way that
$\mathcal{R}^{-1}$ is an orthogonal matrix with rational entries and $N\mathcal{R}^{-1}$
is a matrix with integer entries, where $N$ is a large positive integer depending only on $d$ and $M$.

Let $ w_\varep (y,t)=u_\varep (x, N^2 t)$, where $y=N^{-1} \mathcal{R} x$.
Then
$$
\partial_t w_\varep - \text{div} \big( H(y/\varep, t/\varep^2) \nabla_y w_\varep \big) =\widetilde{F} (y,t)
$$
in $ \big[\Omega_3 \cap B(0, r_0/(2N))\big] \times (-r_0^2/(2N)^2, 0)$, where 
$$
\Omega_3 =\big\{ y\in \mathbb{R}^d:\ y=N^{-1} \mathcal{R} x \text{ for some } x\in \Omega \big\},
$$
$H(y,t)= (h_{ij}^{\alpha\beta} (y,t))$ with 
$$
h_{ij}^{\alpha\beta} (y,t) =\mathcal{R}_{ik} \mathcal{R}_{j\ell}
a_{k\ell}^{\alpha\beta} (N\mathcal{R}^{-1} y, N^2 t),
$$
and $\widetilde{F}(y,t)=N^2 F(N \mathcal{R}^{-1} x, N^2 t)$.
Note that $H(y,t)$ is elliptic and periodic with respect to $\mathbb{Z}^{d+1}$ and
$$
\Omega_3 \cap B(0, r_0/(2N))
=\big\{ (y^\prime, y_d)\in \mathbb{R}^d: \ y_d >\psi_2 (y^\prime) \big\} \cap B(0, r_0/(2N)).
$$
As a result, in the study of uniform boundary estimates for $\partial_t +\mathcal{L}_\varep$,
we may localize the problems to the setting where $\Omega\cap B(P, r_0)$ is given by the region above a graph.

\begin{proof}[\bf Proof of Theorem \ref{theorem-7.1}]
Suppose that $\big(\partial_t +\mathcal{L}_\varep\big) u_\varep =0$ in $\Omega_{2r} (x_0, t_0)$
and $u_\varep =0$ on $\Delta_{2r} (x_0, t_0)$, for
some $x_0\in \partial\Omega$, $t_0\in \mathbb{R}$, and $r>0$ small.
By translation and dilation we may assume that $r=1$ and $(x_0, t_0)=(0,0)$.
We may also assume that $\| u_\varep\|_{L^2(\Omega_2 (0,0))} \le 1$.
By the localization procedure described above we may further assume that
$$
\Omega\cap B(0, 4) =\big\{ (x^\prime, x_d)\in \mathbb{R}^d: \ 
x_d>\psi (x^\prime) \big\} \cap B(0, 4),
$$ 
where $\psi$ is a $C^1$ function in $\mathbb{R}^{d-1}$ satisfying $\psi (0)=0$ and
the condition (\ref{condition-psi}).

Using Theorem \ref{theorem-7.2}, we may deduce that for $0<\rho<1/4$,
$$
\left(\average_{\Omega_\rho (0, 0)} |u_\varep -\average_{\Omega_\rho (0,0)} u_\varep|^2\right)^{1/2}
\le C \rho^\alpha \| u_\varep\|_{L^2(\Omega_1 (0,0))}.
$$
By translation it follows that for any $(x,t)\in \Delta_1 (0,0)=\big[ B(0,1)\cap \partial\Omega\big] \times (-1, 0)$ and
$0<\rho<1/4$,
\begin{equation}\label{7.10-1}
\left(\average_{\Omega_\rho (x, t)} |u_\varep -\average_{\Omega_\rho (x,t)} u_\varep|^2\right)^{1/2}
\le C \rho^\alpha \| u_\varep\|_{L^2(\Omega_2 (0,0))},
\end{equation}
where $C$ depends only on $d$, $m$, $A$, and $\Omega$.
This, together with the interior H\"older estimates, implies that
the estimate (\ref{7.10-1}) in fact holds for any $(x,t)\in \Omega_1 (0,0)$ and $0<\rho<1/4$.
By the Campanato characterization of H\"older spaces we obtain
$$
\| u_\varep\|_{C^{\alpha, \alpha/2} (\Omega_1 (0,0))}
\le C \, \| u_\varep\|_{L^2(\Omega_2 (0,0))}.
$$
This completes the proof of Theorem \ref{theorem-7.1}.
\end{proof}


\section{Boundary $W^{1,p}$ estimates and proof of Theorems \ref{main-theorem-3} and
\ref{main-theorem-4}}

\setcounter{equation}{0}

In this section we study the boundary $W^{1,p}$ estimates for solutions of
$\big(\partial_t +\mathcal{L}_\varep \big) u_\varep =\text{div} (f)$.
We first treat the case $f=0$.

\begin{theorem}\label{theorem-9.1}
Assume that $A$ satisfies conditions (\ref{e1.3})-(\ref{e1.4}) and $A\in VMO_x$.
Let $\Omega$ be a $C^1$ domain in $\mathbb{R}^d$.
Suppose that $\big(\partial_t +\mathcal{L}_\varep \big) u_\varep =0 $
in $\Omega_{2r} (x_0, t_0)$ and $u_\varep=0$ on $\Delta_{2r} (x_0, t_0)$
for some $x_0\in \partial\Omega$ and $t_0\in \mathbb{R}$,
where $0<r< \text{diam}(\Omega)$.
Then, for any $p>2$,
\begin{equation}\label{9.1-1}
\left(\average_{\Omega_{r/4} (x_0, t_0)} |\nabla u_\varep|^p \right)^{1/p}
\le \frac{C_p}{r}
\left(\average_{\Omega_{2r} (x_0, t_0)} |u_\varep|^2 \right)^{1/2},
\end{equation}
where $C_p$ depends at most on $d$, $m$, $A$, and $\Omega$.
\end{theorem}

\begin{proof}
By translation and dilation we may assume that $r=1$ and $(x_0, t_0)=(0,0)$.
It follows from the interior $W^{1,p}$ estimates in Theorem \ref{theorem-10.1}
that if $Q_{2\rho} (x,t)\subset \Omega_2 (0,0)$,
$$
\left(\average_{Q_\rho (x,t)} |\nabla u_\varep|^p \right)^{1/p}
\le \frac{C}{\rho} \left(\average_{Q_{2\rho} (x,t)} |u_\varep|^2\right)^{1/2}.
$$
This implies that if $(x,t)\in \Omega_{1/2} (0,0)$, then
\begin{equation}\label{9.1-3}
\int_{Q_{c\delta(x)} (x,t)} |\nabla u_\varep|^p
\le C \int_{Q_{2c\delta (x)} (x,t)} \left| \frac{u_\varep (y,s)}{\delta (y)}\right|^p \, dyds,
\end{equation}
where $\delta(y) =\text{dist}(y, \partial\Omega)$ and $c>0$ is sufficiently small.
By multiplying both sides of (\ref{9.1-3}) by $[\delta (x)]^{-d-2}$
and integrating the resulting inequality in $(x,t)$ over $\Omega_{1/2}(0, 0)$,
we obtain
\begin{equation}\label{9.1-5}
\average_{\Omega_{1/4} (0, 0)}
|\nabla u_\varep|^p 
\le C \average_{\Omega_1 (0, 0)} \left|\frac{u_\varep  (x,t)} {\delta (x)} \right|^p\, dxdt
\end{equation}
(see \cite[pp.2289-2290]{Shen-2008} for a similar argument in the elliptic case).

Finally, we note that by the boundary H\"older estimates in Theorem \ref{theorem-7.1},
$$
|u_\varep (x,t)|\le C \big[\delta (x)\big]^\alpha 
\| u_\varep\|_{L^2(\Omega_2 (0,0))}, \quad \text{ for any } (x,t)\in \Omega_1(0,0).
$$
This, together with (\ref{9.1-5}), gives
$$
\average_{\Omega_{1/4} (0,0)} |\nabla u_\varep|^p
\le C \average_{\Omega_1 (0,0)} |\delta (x)|^{(\alpha-1) p}\, dxdt \, \| u_\varep\|_{L^2(\Omega_2 (0,0))}^p
\le C\, \|u_\varep\|_{L^2(\Omega_2 (0,0))}^p,
$$
where we have chosen $\alpha\in (0,1)$ so that $(\alpha-1)p>-1$.
The proof is complete.
\end{proof}

\begin{proof}[\bf Proof of Theorem \ref{main-theorem-3}]
Let $2<p<\infty$.
It follows from (\ref{9.1-1}) and Lemma \ref{Poincare-lemma} that, 
if $\big(\partial_t +\mathcal{L}_\varep\big) u_\varep =0$
in $\Omega_{2r} (x_0, t_0)$ and $u_\varep =0$ on $\Delta_{2r} (x_0, t_0)$
for some $x_0\in \partial\Omega$ and $t_0\in \mathbb{R}$, where $0<r<\text{diam}(\Omega)$,
then
\begin{equation}\label{9.2-1}
\left(\average_{\Omega_{r/4} (x_0, t_0)} |\nabla u_\varep |^p\right)^{1/p}
\le C 
\left(\average_{\Omega_{2r} (x_0, t_0)} |\nabla u_\varep |^2\right)^{1/2}.
\end{equation}
Using the interior $W^{1,p}$ estimate in Theorem \ref{theorem-10.1} and some geometric 
consideration, it
is not hard to see that the estimate (\ref{9.2-1}) continues to hold if we replace the assumption
$x_0\in \partial\Omega$ by $x_0 \in \overline{\Omega}$.
The $W^{1,p}$ estimate (\ref{1.3-1}) now  follows from (\ref{9.2-1}) and standard
$L^2$ energy estimates by a real-variable argument, similar to that used in the proof of 
Theorem \ref{main-theorem-1}. We omit the details and refer the reader to \cite{Shen-2005-bounds, Geng} for a similar argument in the elliptic case.

Finally, we point out that if one replaces $Q_r (x_0, t_0)$ by $\Omega_r (x_0, t_0)$ and
$Q_{2r} (x_0, t_0)$ by $\Omega_{8r} (x_0, t_0)$, respectively,
the estimate in Lemma \ref{Poincare-lemma}
continues to hold for solutions of $\big(\partial_t +\mathcal{L} \big) u=\text{div} (f)$
in $\Omega_{8r} (x_0, t_0)$ and $u=0$ on $\Delta_{8r} (x_0, t_0)$, where $x_0\in \overline{\Omega}$.
To see this, we consider two cases: (1) $B(x_0, 2r)\cap \partial\Omega =\emptyset$; (2) 
$B(x_0, 2r)\cap \partial\Omega \neq \emptyset$.
Note that in the first case, where $B(x_0, 2r)\subset \Omega$,
 the desired estimate follows directly from Lemma \ref{Poincare-lemma}.
 In the second case we choose $y_0\in B(x_0, 2r)\cap \partial\Omega$ and use the estimate,
 $$
 \int_{\Omega_r (x_0, t_0)} |u -\average_{\Omega_r (x_0, t_0)} u |^2
  \le  \int_{\Omega_{3r} (y_0, t_0)} |u|^2
 \le C r^2 \int_{\Omega_{3r} (y_0, t_0)} |\nabla u|^2,
 $$
 where the last step follows from the Poincar\'e inequality as well as the assumption $u=0$ on $I_{8r} (x_0, t_0)$.
With these  observations, it is not hard to see that the boundary H\"older estimate (\ref{1.3-2})
follows from (\ref{1.3-1}), as in the interior case.
\end{proof}

\begin{proof}[\bf Proof of Theorem \ref{main-theorem-4}]
Let $u_\varep$ be a weak solution of $\big(\partial_t +\mathcal{L}_\varep \big) u_\varep
=F$ in $\Omega \times (0, T)$ and $u_\varep =0$ on the parabolic boundary
$\big[ \partial\Omega\times (0, T) \big] \cup \big[ \Omega \times \{ 0\}\big]$,
where $F\in L^p(0,T; W^{-1, p} (\Omega))$.
We need to show that
\begin{equation}\label{9.3-1}
\|\nabla u_\varep \|_{L^p(0,T; L^p(\Omega))}
\le C_p\, \| F\|_{L^p (0, T; W^{-1, p}(\Omega))}.
\end{equation}
By a simple duality argument we may assume that $p>2$.

We first consider the case where $F=\text{div} (f)$ for some $f=(f_i)\in L^p(0, T; L^p(\Omega))$.
We extend $u_\varep$ and $F$ by zero to $\Omega \times (-\infty, 0]$
 and $\partial\Omega \times (-\infty, 0]$,
respectively.
Note that since $u_\varep=0$ on the parabolic boundary of $\Omega\times (0,T)$,
$u_\varep$ is a weak solution of $\big(\partial_t +\mathcal{L}_\varep \big)u_\varep
=\text{div} (f) $ in $\Omega \times (-\infty, T)$ and
$u_\varep=0 $ on $\partial\Omega \times (-\infty, T)$.
This allows us to cover the set ${\Omega}\times (0, T)$
by a finite number  of  $\Omega_{r} (x_\ell, t_\ell)$ with the properties that $(x_\ell, t_\ell)\in
\overline{\Omega}\times [0, T]$ and  $r=c_0 \min (\text{diam}(\Omega), \sqrt{T})$, 
and apply the $W^{1,p}$  estimates in Theorems \ref{main-theorem-1} and \ref{main-theorem-3}
on each $\Omega_r (x_\ell, t_\ell)$.
It follows by summation that
$$
\|\nabla u_\varep\|_{L^p(0, T; L^p(\Omega))}
\le C \left\{ \| u_\varep\|_{L^2(0,T; L^2(\Omega))}
+\| f\|_{L^p(0, T; L^p(\Omega))} \right\}
\le C\| f\|_{L^p(0, T; L^p(\Omega))},
$$
where we have used the energy estimates as well as H\"older's inequality for the last inequality.

Finally, we note that if $F\in L^p(0, T; W^{-1, p}(\Omega))$, then $F =g +\text{div} (f)$, where
$g, f\in L^p(0, T; L^p(\Omega))$ and 
$$
\| g\|_{L^p(0, T;L^p (\Omega))}
+\| f\|_{L^p(0, T; L^p(\Omega))} \le C \| F\|_{L^p(0, T; W^{-1, p}(\Omega))}.
$$
Let $w$ be the solution of the heat equation $\big(\partial_t -\Delta \big) w =g$ in $\Omega\times (0, T)$
and $w=0$ on the parabolic boundary. Then
$$
\big(\partial_t +\mathcal{L}_\varep \big) (u_\varep -w) 
=\text{div} (f ) -\Delta w -\mathcal{L}_\varep (w).
$$
It follows that
$$
\aligned
\| \nabla u_\varep\|_{L^p(0, T; L^p(\Omega))} 
&\le \| \nabla (u_\varep -w )\|_{L^p(0, T; L^p(\Omega))} 
+\| \nabla w\|_{L^p(0, T; L^p(\Omega))}\\
& \le C \left\{ \| f\|_{L^p(0, T; L^p(\Omega))}
+\|\nabla w\|_{L^p(0, T; L^p(\Omega))}\right\} \\
& \le C \left\{ \| g\|_{L^p(0, T; L^p(\Omega))} 
+\| f\|_{L^p(0, T; L^p(\Omega))} \right\},
\endaligned
$$
where we have used the $W^{1,p}$ estimates for the heat equation in $C^1$ cylinders  in the last inequality.
 This completes the proof.
\end{proof}

\begin{remark}
{\rm 
By subtracting from $u_\varep$ a solution of $(\partial_t -\Delta ) w=0$ in $\Omega\times (0, T)$ with boundary data
$w=h$ on $\partial\Omega\times (0, T)$, one may handle the $W^{1,p}$ estimates
for the nonhomogenous initial-Dirichlet problem 
$\big (\partial_t +\mathcal{L} \big) u_\varep =F$ in $\Omega\times (0, T)$,
$u_\varep=h$ on $\partial\Omega \times (0, T)$, and $u_\varep =0$ on $\Omega \times \{ 0\}$.
We leave the details to the interested readers.
}
\end{remark}


\bibliography{gs1.bbl}

\medskip

\begin{flushleft}
Jun Geng, 
School of Mathematics and Statistics, 
Lanzhou University,
Lanzhou, P. R. China.

E-mail: gengjun@lzu.edu.cn
\end{flushleft}

\begin{flushleft}
Zhongwei Shen, 
 Department of Mathematics,
University of Kentucky,
Lexington, Kentucky 40506,
USA. 

E-mail: zshen2@uky.edu
\end{flushleft}

\medskip

\noindent \today

\end{document}